\documentclass[12pt]{amsart}

\usepackage{amsmath}
\usepackage{amssymb}
\usepackage{amscd}
\usepackage[cmtip,all]{xy}

\title{Quasi-Galois points, II: Arrangements}
\author{Satoru Fukasawa, Kei Miura and Takeshi Takahashi}

\address{Faculty of Science, Yamagata University,  
Kojirakawa-machi 1-4-12, Yamagata 990-8560, Japan}
\email{s.fukasawa@sci.kj.yamagata-u.ac.jp} 
\thanks{The work of the first author was partially supported by JSPS KAKENHI Grant Numbers 25800002, JP16K05088 and JP19K03438.} 
\address{Department of Mathematics, National Institute of Technology, Ube College, Ube, Yamaguchi 755-8555, Japan}
\email{kmiura@ube-k.ac.jp}
\thanks{The work of the second author was partially supported by JSPS KAKENHI Grant Numbers 26400057 and JP18K03230.} 
\address{Education Center for Engineering and Technology, Faculty of Engineering, Niigata University, Niigata 950-2181, Japan}
\email{takeshi@eng.niigata-u.ac.jp}
\thanks{The work of the third author was partially supported by JSPS KAKENHI Grant Numbers 25400059, JP16K05094 and JP19K03441.}

\newtheorem{theorem}{Theorem}[section]
\newtheorem{proposition}[theorem]{Proposition}
\newtheorem{corollary}[theorem]{Corollary}
\newtheorem{lemma}[theorem]{Lemma} 

\newtheorem{fact}[theorem]{Fact}

\theoremstyle{definition}

\newtheorem{remark}[theorem]{Remark}

\begin{document}
\begin{abstract}  
In Part I, the present authors introduced the notion of a quasi-Galois point, for investigating the automorphism groups of plane curves.
In this second part, the number of quasi-Galois points for smooth plane curves is described. 
In particular, sextic or quartic curves with many quasi-Galois points are characterized. 
\end{abstract}
\maketitle

\section{Introduction}  
In Part I \cite{fmt}, the present authors introduced the notion of a {\it quasi-Galois point} for a plane curve $C \subset \mathbb{P}^2$, for investigating the automorphism group ${\rm Aut}(C)$ of $C$. 
In this second part, we describe the arrangement of quasi-Galois points. 
It is inferred that quasi-Galois points are useful to classify algebraic curves. 

Let $C \subset \mathbb{P}^2$ be a smooth plane curve of degree $d \ge 4$ over an algebraically closed field $K$ of characteristic zero, and let $P \in \mathbb{P}^2$. 
We define the set $G[P]$ as the group consisting of all birational transformations of $C$ preserving the fibers of the projection $\pi_P$. 
If $|G[P]| \ge 2$, then we say that $P$ is a {\it quasi-Galois point}. 
This is a generalization of the Galois point, which was introduced by Hisao Yoshihara (\cite{fukasawa, miura-yoshihara, yoshihara}).

In this second part, the number of quasi-Galois points for smooth plane curves is determined. 
The number $\delta[n]$ (resp., $\delta'[n]$) of quasi-Galois points $P \in C$ (resp., $P \in \mathbb{P}^2 \setminus C$) with $|G[P]|=n$ is determined explicitly for any $n \ge 3$, in Theorems \ref{inner}, \ref{outer} and \ref{outer n=3}. 
Furthermore, when $d=4$ or $6$, all possibilities of $\delta'[d/2]$ are determined (Theorems \ref{sextic} and \ref{quartic}).

Quasi-Galois points are related to reflections or finite unitary reflection groups in group theory. 
In fact, a generator of the associated group $G[P]$ at a quasi-Galois point is represented by a reflection (see \cite[Theorem 2.3]{fmt} or Fact \ref{standard form}). 
Finite unitary reflection groups are well studied, and were classified in 1950s when such groups are ``irreducible'' (Shephard--Todd \cite{shephard-todd}, \cite[Theorem 8.29]{lehrer-taylor}).  
Proofs of our results do not depend on the results of them. 
Some parts of our proofs of Theorems \ref{outer} and \ref{outer n=3} are related to the methods of Mitchell in \cite[Sections 5 and 6]{mitchell}. 
Our proofs come from the point of view of algebraic geometry.

\section{Preliminaries} 
We introduce the system $(X:Y:Z)$ of homogeneous coordinates on $\mathbb P^2$. 
If $P \in C$, then the (projective) tangent line at $P$ is denoted by $T_PC$. 
For a projective line $\ell \subset \mathbb P^2$ and a point $P \in C \cap \ell$, the intersection multiplicity of $C$ and $\ell$ at $P$ is denoted by $I_P(C, \ell)$.  
The line passing through points $P$ and $Q$ is denoted by $\overline{PQ}$, when $P \ne Q$, and the projection from a point $P \in \mathbb P^2$ by $\pi_P$, which is the rational map from $C$ to $\mathbb P^1$ represented by $Q \mapsto \overline{PQ}$. 
If $Q \in C$, the ramification index of $\pi_P$ at $Q$ is denoted by $e_Q$.   
We note the following elementary fact.  
 
\begin{fact} \label{index}
Let $P \in \mathbb P^2$, and let $Q \in C$. 
Then, for $\pi_P$ we have the following.
\begin{itemize}
\item[(1)] If $P=Q$, then $e_P=I_P(C, T_PC)-1$.  
\item[(2)] If $P \ne Q$, then $e_Q=I_Q(C, \overline{PQ})$.   
\end{itemize} 
\end{fact}

If $|G[P]| \ge 2$, then the fixed field $K(C)^{G[P]}$ is an intermediate field of $K(C)/\pi_P^*K(\mathbb{P}^1)$, and we have a Galois covering $C \rightarrow C/G[P]$. 

\begin{remark}
The order $|G[P]|$ divides the degree of $\pi_P$. 
\end{remark}

In general, the following fact holds for a Galois covering $\theta:C \rightarrow C'$ with a Galois group $G$ between smooth curves, where $G(P)$ is the stabilizer subgroup of $P$ (see \cite[III. 7.2, 8.2]{stichtenoth}). 
\begin{fact} \label{Galois covering} 
Let $\theta: C \rightarrow C'$ be a Galois covering of degree $d$, and let $G$ be the Galois group. 
Then: 
\begin{itemize}
\item[(1)] The order of $G(P)$ is equal to $e_P$ at $P$ for any point $P \in C$.\item[(2)] Let $P, Q \in C$. If $\theta(P)=\theta(Q)$, then $e_P=e_Q$. 
\end{itemize} 
\end{fact}

Note that any automorphism is the restriction of a linear transformation (see \cite[Appendix A, 17 and 18]{acgh} or \cite{chang}), since $C$ is smooth and of degree $d \ge 4$. 
According to Part I \cite[Remark 2.2 and Theorem 2.3]{fmt}, we have the following fact and two corollaries. 

\begin{fact}[\cite{fmt}, Theorem 2.3] \label{standard form} 
The group $G[P]$ is a cyclic group.  
Furthermore, for an integer $n \ge 2$, $n$ divides $|G[P]|$ if and only if there exists a linear transformation $\phi$ such that 
\begin{itemize}
\item[(1)] $\phi(P)=(1:0:0)$,  
\item[(2)] there exists an element $\sigma \in G[\phi(P)] \subset {\rm Bir}(\phi(C))$ which is represented by the matrix 
$$ A_{\sigma}=\left(\begin{array}{ccc} \zeta & 0 & 0 \\ 0 & 1 & 0 \\
0 & 0 & 1 \end{array} \right), $$
where $\zeta$ is a primitive $n$-th roof of unity, and 
\item[(3)] $\phi(C)$ is given by 
$$ \sum_i G_{d-n i}(Y, Z)X^{n i}=0, $$
where $G_{d-n i}$ is a homogeneous polynomial of degree $d-n i$ in variables $Y, Z$. 
\end{itemize}
\end{fact}

\begin{corollary} \label{fixed locus} 
For $\sigma \in G[P] \setminus \{1\}$, we define $F[P]:=\{Q \in \mathbb P^2 \ | \  \sigma(Q)=Q \}$. 
If we use the standard form as in Fact \ref{standard form}, $F[P]=\{P\} \cup \{X=0\}$. 
In particular, the set $F[P]$ does not depend on $\sigma$. 
\end{corollary} 

\begin{corollary} \label{two groups}
Let $P_1, P_2 \in \mathbb{P}^2$. 
If $P_1 \ne P_2$, then $G[P_1] \cap G[P_2]=\{1\}$. 
\end{corollary}

We will use the following two well-known facts. 

\begin{fact} \label{total-cyclic}
Let $G \subset {\rm Aut}(C)$ be a finite subgroup, and let $Q \in C$ be a point.  
If $\sigma(Q)=Q$ for any $\sigma \in G$, then $G$ is a cyclic group. 
\end{fact} 

\begin{fact} \label{PGL}
Let $G$ be a finite subgroup of ${\rm PGL}(2, K)$. 
Then $G$ is isomorphic to one of the following: 
\begin{itemize}
\item[(1)] a cyclic group; 
\item[(2)] a dihedral group; 
\item[(3)] the alternating group $A_4$ of degree four; 
\item[(4)] the symmetric group $S_4$ of degree four; 
\item[(5)] the alternating group $A_5$ of degree five. 
\end{itemize}
\end{fact}

To study the number of quasi-Galois points, we introduce some symbols here. 
The set of all quasi-Galois points $P \in C$ with $|G[P]|=n$ (resp., $|G[P]| \ge n$) is denoted by $\Delta_n$ (resp., $\Delta_{\ge n}$). 
The number of quasi-Galois points $P \in C$ with $|G[P]|=n$ (resp., $|G[P]| \ge n$) is denoted by $\delta[n]$ (resp., $\delta[\ge n]$). 
Similarly, we define $\Delta'_n$, $\Delta'_{\ge n}$, $\delta'[n]$ and $\delta'[\ge n]$, when we consider the case $P \in \mathbb P^2 \setminus C$. 

\section{The number of quasi-Galois points} 

Let $P \in \mathbb P^2$ be a quasi-Galois point for $C$ with $|G[P]|=n \ge 2$. 
We consider ramification points for the projection $\pi_P$. 

\begin{proposition} \label{tangent1} 
There exist $d$ points $Q_1, \ldots, Q_d \in C \cap (F[P]\setminus \{P\})$ such that $P \in T_{Q_i}C$ and $I_{Q_i}(C, T_{Q_i}C)=l_i n$ for some integer $l_i \ge 1$. 
\end{proposition}

\begin{proof} 
Let $Q \in C \cap (F[P] \setminus \{P\})$. 
By Corollary \ref{fixed locus}, $\sigma(Q)=Q$ for each $\sigma \in G[P]$. 
By Fact \ref{Galois covering}(1), the ramification index at $Q$ for  the covering map $C \mapsto C/G[P]$ is equal to $n$. 
Since the projection $\pi_P$ is the composite map of $C \rightarrow C/G[P]$ and $C/G[P] \rightarrow \mathbb P^1$, the ramification index $e_Q$ at $Q$ for $\pi_P$ is equal to $l n$ for some $l \ge 1$. 
By Fact \ref{index}(2), $e_Q=I_Q(C, \overline{PQ})=l n$ and $\overline{PQ}=T_QC$. 
Furthermore, the line given by $F[P]\setminus \{P\}$ consists of exactly $d$ points. 
\end{proof}

If $P \in C$, we have the following. 

\begin{proposition} \label{tangent2}
If $P \in C$, then $I_P(C, T_PC)=l n+1$ for some integer $l \ge 1$. 
\end{proposition}

\begin{proof}
By Corollary \ref{fixed locus}, for any $\sigma \in G[P]$, $\sigma(P)=P$. 
Then the covering map $C \rightarrow C/G[P]$ is ramified at $P$ with index $n$, by Fact \ref{Galois covering}(1). 
Since the projection $\pi_P$ is the composite map of $C \rightarrow C/G[P]$ and $C/G[P] \rightarrow \mathbb P^1$, the ramification index $e_P$ at $P$ is equal to $l n$ for some $l \ge 1$. 
Note that $e_P=I_P(C, T_PC)-1$, by Fact \ref{index}(1). 
It follows that $I_P(C, T_PC)=l n+1$. 
\end{proof}

Using Fact \ref{total-cyclic}, we have the following. 

\begin{proposition}\label{two quasi-Galois}
Let $P_1, P_2\in \mathbb P^2$ be points with $|G[P_1]|=n_1 \ge 2$, $|G[P_2]|=n_2 \ge 2$. 
\begin{itemize}
\item[(1)] If $P_1, P_2 \in C$, then $C \cap F[P_1] \cap F[P_2] \subset \{P_1, P_2\}$. 
Furthermore, if $n_1$ and $n_2$ are not coprime, then $C \cap F[P_1] \cap F[P_2]=\emptyset$. 
\item[(2)] If $P_1, P_2 \in \mathbb{P}^2 \setminus C$, then $C \cap F[P_1] \cap F[P_2] = \emptyset$. 
\end{itemize} 
\end{proposition}

\begin{proof}
Assume that there exists a point $Q \in C \cap F[P_1] \cap F[P_2]$. 
Note that, by definition, points $P_1, P_2$ and $Q$ are collinear. 

First, we assume that $n_1$ and $n_2$ are divisible by some integer $n \ge 2$. 
Since $G[P_1]$ and $G[P_2]$ are cyclic by Fact \ref{standard form}, there exist subgroups of $G[P_1]$ and $G[P_2]$ of order $n$ respectively. 
Let $G$ be the group generated by such subgroups.  
Then $G$ fixes the point $Q$. 
By Fact \ref{total-cyclic}, $G$ is a cyclic group. 
Therefore, by Corollary \ref{two groups}, $G$ is a cyclic group of order $n^2$. 
However, the cyclic group of order $n^2$ has a unique subgroup of order $n$. 
This is a contradiction. 
In particular, the latter assertion of (1) follows. 

Next, we consider the case where $Q \ne P_1, P_2$. 
Let $\sigma \in G[P_1] \setminus \{1\}$. 
Since $\sigma$ fixes $P_1$ and $Q$ on the line $\overline{P_1Q}=\overline{P_2Q}$, it follows that $P_3:=\sigma(P_2) \ne P_2$. 
Then $G[P_3]=\sigma G[P_2] \sigma^{-1}$ and $Q \in C \cap F[P_2] \cap F[P_3]$. 
By the above discussion, we have a contradiction. 
Assertions (1) and (2) follow. 
\end{proof}

For the number of quasi-Galois points on $C$, we have the following.

\begin{theorem} \label{inner} Let $n \ge 3$.  
Then
$$ \delta[n]=0, 1 \mbox{ or } 4. $$ 
Furthermore, $\delta[n]=4$ only if $n=3$, and $d=6m+4$ for some integer $m \ge 0$. 
\end{theorem} 

\begin{proof}
Let $P_1$ and $P_2 \in C$ be quasi-Galois points with $|G[P_1]|=|G[P_2]|=n$, and let $\ell=\overline{P_1P_2}$. 
Note that $\sigma(\ell)=\ell$ for each $\sigma \in G[P_i]$ for $i=1,2$. 
Let 
$$G :=\{ \sigma \in {\rm Aut}(C) \ | \ \sigma(\ell)=\ell\} \subset {\rm PGL}(3, K), $$
and let $\varphi: G \rightarrow {\rm Aut}(\ell) \cong {\rm PGL}(2, K)$ be the homomorphism defined by $\sigma \mapsto \sigma |_\ell$. 
Since $\sigma(P_2) \ne P_2$ for each element $\sigma \in G[P_1] \setminus \{1\}$, we have $mn+1$ quasi-Galois points $P_1, P_2, \ldots, P_{mn+1}$ on the line $\ell$ for some integer $m$.  
Note that the restriction of $\varphi$ over $G[P_i]$ is injective for each $i$. 
By Fact \ref{PGL}, $\varphi(G)=A_4$, $S_4$ or $A_5$. 
Since the stabilizer subgroup $\varphi(G)(P_i)$ of $\varphi(G)$ acts on the projective line $\ell$, $\varphi(G)(P_i)$ is a cyclic group such that
$$ n \le |\varphi(G)(P_i)| \le 5, $$
for each $i$. 

Assume that $n=5$. 
Then $|\varphi(G)(P_i)|=5$ and $\varphi(G)\cong A_5$. 
Since $\varphi(G)(P_i)$ is a Sylow $5$-group, $\varphi(G)$ acts on the set $\{P_i\}$ transitively. 
Therefore, 
$$ 5(5m+1)=60 $$
holds. 
This is a contradiction. 

Assume that $n=4$. 
Then $|\varphi(G)(P_i)|=4$ and $\varphi(G) \cong S_4$.  
Note that $S_4$ has exactly three cyclic subgroups of order $4$. 
Since $\varphi(G)$ has at least $5$ cyclic subgroups of order $4$, this is a contradiction.  

Assume that $n=3$. 
Then $|\varphi(G)(P_i)|=3$.  
Since $\varphi(G)(P_i)$ is a Sylow $3$-group, $\varphi(G)$ acts on the set $\{P_i\}$ transitively. 
Therefore, 
$$ 3(3m+1)=12, 24 \mbox{ or } 60. $$
This implies that $m=1$.  

We have to show that $\ell$ is a unique line containing exactly four quasi-Galois points on $C$, in the case where $n=3$. 
By Lemma \ref{four quasi-Galois} below, it is inferred that for each four quasi-Galois points on a line $\ell$, there exists a quasi-Galois point $Q$ such that $\ell=F[Q] \setminus \{Q\}$.  
Assume that $\delta[3] \ge 5$. 
Then there exist two lines $\ell$ and $\ell'$ containing four quasi-Galois points such that the point $P \in \ell \cap \ell'$ is a quasi-Galois point on $C$. 
Then there exist two quasi-Galois points $Q$ and $Q'$ such that $\ell=F[Q] \setminus \{Q\}$ and $\ell'=F[Q'] \setminus \{Q'\}$.  
This implies that $P \in F[Q] \cap F[Q']$, and hence, this is a contradiction to Proposition \ref{two quasi-Galois}(2). 
\end{proof}

\begin{lemma} \label{four quasi-Galois}
Let $\ell$ be a line containing four points $P_1, P_2, P_3$ and $P_4 \in \mathbb{P}^2$ with $|G[P_i]|=3$ for each $i$. 
If the group $\langle G[P_1], G[P_2] \rangle$ acts on the set $\{P_1, P_2, P_3, P_4\}$, 
then there exists a point $Q \notin \ell$ such that $|G[Q]|=2m$ for some integer $m \ge 1$, and $F[Q]\setminus \{Q\}=\ell$. 
Furthermore, $Q \in \mathbb{P}^2 \setminus C$ and $d$ is even.  
\end{lemma}

\begin{proof}
Let $\omega^2+\omega+1=0$. 
We can take a system of coordinates so that $P_1=(1:0:0)$, $P_2=(1:-1:0)$, $P_3=(1:-\omega^2:0)$ and $P_4=(1:-\omega:0)$, and a generator of $G[P_1]$ is represented by 
$$ \sigma_1=\left(\begin{array}{ccc}
\omega & 0 & 0 \\
0 & 1 & 0 \\
0 & 0 & 1 
\end{array}\right).
$$
Furthermore, we can assume that the coordinate of the point $Q$ given by $F[P_1] \cap F[P_2]$ is $(0:0:1)$.  
Let $\sigma_2 \in G[P_2]$ be a generator. 
By the condition that $\sigma_2(Q)=Q$, $\sigma_2(P_2)=P_2$, $\sigma_2(P_1)=P_4$, $\sigma_2(P_3)=P_1$,  
it follows that $\sigma_2$ is represented by
$$\left(\begin{array}{ccc} 
-\omega \alpha & 2\omega^2 \alpha & 0 \\
\omega^2 \alpha & \alpha & 0\\
0 & 0 & 1
\end{array}\right) $$
for some $\alpha \in K$. 
By $\sigma_2^*(x+y)=x+y$, $\alpha=1/(2\omega^2+1)$ follows. 
Note that $\alpha^2=-1/3$. 
Then it follows that
$$ (\sigma_2\sigma_1^2)^2=
\left(\begin{array}{ccc}
-1 & 0 & 0 \\
0 & -1 & 0 \\
0 & 0 & 1 \\
\end{array}\right),  $$
and hence, $G[Q]$ contains an element of order $2$. 
Therefore, $|G[Q]|$ is even and $F[Q] \setminus \{Q\}=\ell$. 
If $Q \in C$, the tangent line at $Q$ contains $P_1$ and $P_2$. 
This is a contradiction. 
Therefore, $Q \in \mathbb{P}^2 \setminus C$ and $d$ is even.  
\end{proof}

\begin{corollary}
We have 
$$ \delta[\ge 3]=0, 1, 2 \mbox{ or } 4. $$
Furthermore, $\delta[\ge 3]=4$ only if $\delta[\ge 3]=\delta[3]=4$. 
\end{corollary}

\begin{proof}
Assume that $\delta[\ge 3] \ge 3$ and $\delta[\ge 4] \ge 1$. 
Let $P_1, P_2, P_3 \in C$ be different points with $|G[P_1]| \ge 3$, $|G[P_2]| \ge 3$ and $|G[P_3]| \ge 4$. 
It follows from Theorem \ref{inner} that $P_3 \in F[P_1] \cap F[P_2]$. 
By Lemma \ref{index}(2), $P_1, P_2 \in T_{P_3}C$. 
In this case, points $P_1, P_2$ and $P_3$ are collinear. 
By Corollary \ref{fixed locus}, $P_1 \not\in F[P_3]$ or $P_2 \not\in F[P_3]$. 
Assume that $P_1 \not\in F[P_3]$. 
In this case, there exists a point $P_1' \in C$ such that $|G[P_1']|=|G[P_1]|$ and $P_3 \in F[P_1']$. 
By Proposition \ref{two quasi-Galois}, this is a contradiction. 
\end{proof}

We consider the number of quasi-Galois points in $\mathbb{P}^2 \setminus C$. 
To do this, we introduce the notion of ``$G$-pairs''. 
Let $P, P' \in \mathbb{P}^2 \setminus C$ be points such that $P \ne P'$ and $|G[P]|$ and $|G[P']|$ is divisible by $n \ge 2$.  
We call the pair $(P, P')$ a {\it $G$-pair with respect to $n$} if $\sigma(P')=P'$ and $\sigma'(P)=P$ for generators $\sigma \in G[P]$ and $\sigma' \in G[P']$. 
By Corollary \ref{fixed locus}, the definition does not depend on the choice of generators. 

\begin{lemma} \label{pair 1}
Let $n \ge 2$, let $P_1, P_2 \in \mathbb{P}^2 \setminus C$ be different points such that $n$ divides $G[P_1]$ and $G[P_2]$, and let $\sigma_i \in G[P_i]$ be a generator  for $i=1, 2$. 
If $\sigma_1(P_2)=P_2$, then $\sigma_2(P_1)=P_1$. 
In particular, $(P_1, P_2)$ is a $G$-pair with respect to $n$. 
\end{lemma} 

\begin{proof} 
By the assumption, $P_2 \in F[P_1]\setminus \{P_1\}$. 
It follows from Corollary \ref{fixed locus} and Proposition \ref{tangent1} that the set $F[P_1]\setminus \{P_1\}$ is a line containing $d$ points $Q_1, \ldots, Q_d \in C$ with $\overline{P_1Q_i}=T_{Q_i}C$ for each $i$.
Since $F[P_1] \setminus \{P_1\}$ is a line passing through $P_2$, it follows that $\sigma_2(Q_1)=Q_i$ and $\sigma_2(Q_2)=Q_j$ for some $i, j$. 
Since $\overline{P_1Q_1}$ and $\overline{P_1Q_i}$ are tangent lines at $Q_1$ and $Q_i$ respectively, $\sigma_2(\overline{P_1Q_1})=\overline{P_1Q_i}$. 
Then $\sigma_2(\overline{P_1Q_1} \cap \overline{P_1Q_2}) \subset \overline{P_1Q_i} \cap \overline{P_1Q_j}=\{P_1\}$. 
It follows that $\sigma_2(P_1)=P_1$.  
\end{proof} 

\begin{proposition} \label{pair 2}
There exists a $G$-pair $(P, P')$ with respect to $n$, if and only if $C$ is projectively equivalent to the curve defined by  
$$ g(x^n, y^n)=0 $$
for some polynomial $g$. 
In this case, there exists a point $P'' \in \mathbb{P}^2 \setminus (C \cup \overline{PP'})$ such that pairs $(P, P'')$ and $(P', P'')$ are $G$-pairs. 
In particular, $\delta'[\ge n] \ge 3$. 
\end{proposition}

\begin{proof} 
We consider the if part. 
According to Fact \ref{standard form}, for the defining equation $g(x^n, y^n)=0$, it follows that $P=(1:0:0)$ and $P'=(0:1:0)$ form a $G$-pair with respect to $n$.

We prove the only-if part. 
Assume that $(P, P')$ be a $G$-pair with respect to $n$. 
By the assumption, $P' \in F[P]$ and $P \in F[P']$. 
By Fact \ref{standard form}, for a suitable system of coordinates, we can assume that $P=(1:0:0)$ and there exists an element $\sigma \in G[P]$ of order $n$ which is represented by the matrix 
$$A_{\sigma}=\left(\begin{array}{ccc} \zeta & 0 & 0 \\ 0 & 1 & 0 \\ 0 & 0 & 1 \end{array}\right),  
$$ 
where $\zeta$ is a primitive $n$-th root of unity. 
Then the line given by $F[P]\setminus \{P\}$ is defined by $X=0$. 
Since $P' \in F[P] \setminus \{P\}$, $P'=(0:1:a)$ for some $a \in K$. 
If we take the linear transformation $(X:Y:Z) \mapsto (X:Y:Z-aY)$, we can assume that $P'=(0:1:0)$. 
Then there exists an element $\sigma \in G[P']$ of order $n$ which is represented by the matrix 
$$A_{\sigma'}=\left(\begin{array}{ccc} 1 & 0 & 0 \\ a & \zeta & b \\ 0 & 0 & 1 \end{array}\right), $$ 
for some $a,b \in K$.  
Since the line given by $F[P']\setminus \{P'\}$ is defined by $aX+(\zeta-1)Y+bZ=0$ and $P \in F[P'] \setminus \{P'\}$, it follows that $a=0$. 
If we take 
$$ B=\left(\begin{array}{ccc} 1-\zeta & 0 & 0 \\ 0 & 1 & b \\ 0 & 0 & 1-\zeta \end{array}\right), $$
then 
$$ B^{-1}A_{\sigma_1}B=\left(\begin{array}{ccc} \zeta & 0 & 0 \\ 0 & 1 & 0 \\ 0 & 0 & 1 \end{array}\right), \ B^{-1}A_{\sigma_2}B=\left(\begin{array}{ccc} 1 & 0 & 0 \\ 0 & \zeta & 0 \\ 0 & 0 & 1 \end{array}\right).  $$
By taking the linear transformation represented by $B^{-1}$, the defining polynomial of $C$ is of the form  
$$ g(x^n, y^n)=0. $$
The assertion follows. 

In this case, the automorphism   
$$ 
\left(\begin{array}{ccc}
\zeta & 0  & 0 \\
0 & 1 & 0 \\
0 & 0 & 1 \\
\end{array} \right)
\times
\left(\begin{array}{ccc}
1 & 0  & 0 \\
0 & \zeta & 0 \\
0 & 0 & 1 \\
\end{array} \right)
=
\left(\begin{array}{ccc}
\zeta & 0  & 0 \\
0 & \zeta & 0 \\
0 & 0 & 1 \\
\end{array} \right)
\sim 
\left(\begin{array}{ccc}
1 & 0  & 0 \\
0 & 1 & 0 \\
0 & 0 & \zeta^{-1} \\
\end{array} \right)
$$
acts on $C$. 
Then the point $P''=(0:0:1)$ is a quasi-Galois point with $|G[P'']| \ge n$. 
We have $\delta'[\ge n] \ge 3$. 
\end{proof}

\begin{corollary} \label{pair 2'}
Let $d$ be even, $n=d/2$ and let $(P, P')$ be a $G$-pair.   
Then $C$ is projectively equivalent to the curve defined by  
$$ X^{2n}+Y^{2n}+Z^{2n}+aX^{n}Y^{n}+bY^nZ^n+cZ^nX^n=0, $$
where $a, b, c \in K$. 
\end{corollary}

\begin{proof}
Since $n=d/2$, by Proposition \ref{pair 2}, the defining polynomial of $C$ is of the form  
$$ F=X^{2n}+(aY^n+bZ^n)X^n+(\alpha Y^{2n}+\beta Y^nZ^n+\gamma Z^{2n}). $$
We can assume $\alpha=\gamma=1$. 
\end{proof}

We consider the case where there exist two quasi-Galois points $P_1, P_2 \in \mathbb P^2 \setminus C$. 
Let $\ell=\overline{P_1P_2}$, let
$$G :=\{ \sigma \in {\rm Aut}(C) \ | \ \sigma(\ell)=\ell \} \subset {\rm PGL}(3, K), $$
and let $\varphi: G \rightarrow {\rm Aut}(\overline{P_1P_2}) \cong {\rm PGL}(2, K)$ be the homomorphism defined by $\sigma \mapsto \sigma |_\ell$. 
Note that $\sigma(\ell)=\ell$ for each $\sigma \in G[P_i]$, and the induced homomorphism $G[P_i] \rightarrow \varphi(G[P_i])$ is injective, for $i=1, 2$.

\begin{theorem} \label{outer}
Assume that $n \ge 4$, and points $P_1$ and $P_2 \in \mathbb{P}^2 \setminus C$ are different quasi-Galois points with $|G[P_1]|=|G[P_2]|=n$.
Then there exists a point $P_1' \in \ell:=\overline{P_1P_2}$ such that $(P_1, P_1')$ is a $G$-pair with respect to $n$. 
Furthermore, the following hold. 
\begin{itemize}
\item[(1)] If $n \ge 6$, then $\delta'[\ge n]=3$, and each two of the three quasi-Galois points form a $G$-pair. 
Furthermore, $\delta'[\ge 3]=\delta'[\ge n]=3$. 
\item[(2)] If $n=5$, then $\#\Delta_5' \cap \ell=2$ or $12$. 
Furthermore, if $\#\Delta_5' \cap \ell=12$, then $\#\Delta_{\ge 5}' \cap (\mathbb{P}^2 \setminus \ell)=\#\Delta_{\ge 10}' \cap (\mathbb{P}^2 \setminus \ell)=1$. 
In particular, $\delta'[5]=2, 3$ or $12$. 
\item[(3)] If $n=4$, then $\#\Delta_4' \cap \ell=2$ or $6$. 
Furthermore, if $\#\Delta_4' \cap \ell=6$, then $\#\Delta_{\ge 4}' \cap (\mathbb{P}^2 \setminus \ell) =1$. 
In particular, $\delta'[4]=2, 3, 6$ or $7$. 
\end{itemize} 
\end{theorem}

\begin{proof}
Assume that $(P_1, P_2)$ is not a $G$-pair.   
Since $\sigma(P_2) \ne P_2$ for each element $\sigma \in G[P_1] \setminus \{1\}$, there exist at least $n+1$ quasi-Galois points $P_1, P_2, \ldots, P_{n+1}$ on the line $\ell$. 
Since $\varphi(G)$ contains at least $(n+1)/2 \ge 2$ subgroups of order $n \ge 3$, by Fact \ref{PGL}, $\varphi(G)=A_4$, $S_4$ or $A_5$. 
Then $n \le 5$. 

Assume that $n \ge 6$. 
Then $(P_1, P_2)$ is a $G$-pair. 
By Proposition \ref{pair 2}, there exists a point $P_3$ such that $(P_1, P_2)$, $(P_2, P_3)$, $(P_3, P_1)$ are $G$-pairs. 
If there exists a point $Q \not\in \{P_1, P_2, P_3\}$ with $|G[Q]| \ge 3$ and $Q \in \mathbb{P}^2 \setminus C$, then $P_i \not\in F[Q]$ for some $i$. 
Then there exists a pair $(P_i, P_i')$ with $|G[P_i]|=|G[P_i']| \ge 6$ such that $(P_i, P_i')$ is not a $G$-pair on the line $\overline{QP_i}$. 
This is a contradiction. 
It follows that $\delta'[\ge 3]=\delta'[\ge n]=3$. 
Hereafter, for the case where $n=3, 4$ or $5$, we can assume that $\delta'[m n] \le 1$ for any $m \ge 2$. 

Let $n=5$. 
Assume that there does not exist a $G$-pair on the line $\ell$. 
Then there exist $5m+1$ subgroups of $\varphi(G) \cong A_5$ of order five for some integer $m$. 
Such groups are Sylow $5$-groups and hence, $5(5m+1)=60$. 
This is a contradiction. 
Therefore, there exists a $G$-pair $(P, P')$ on the line $\ell$. 
By Proposition \ref{pair 2}, there exists a quasi-Galois point $P''$ with $F[P''] \setminus \{P''\}=\ell$. 
Since $P_1 \in \ell=F[P''] \setminus \{P''\}$, it follows from Lemma \ref{pair 1} that $(P_1, P'')$ is a $G$-pair. 
By Proposition \ref{pair 2}, there exists a quasi-Galois point $P_1' \in \ell \cap F[P_1]$. 
Then $(P_1, P_1')$ is a $G$-pair. 
In particular, $\#\Delta_5' \cap \ell$ is even. 
If $\#\Delta_5 \cap \ell \ge 3$, then $\varphi(G) \cong A_5$. 
Since there exist exactly six subgroups of $A_5$ of order $5$, we have exactly $12$ quasi-Galois points on the line $\ell$.  

We consider the case where $\#\Delta'_5 \cap \ell=12$. 
Then $\varphi(G) \cong A_5$. 
Note that $(P_1, P_1')$, $(P_1', P'')$ and $(P'', P_1)$ are $G$-pairs, and $F[P'']\setminus \{P''\}=\ell$. 
By Lemma \ref{n=5} below, $|G[P'']|=10m$ for some $m \ge 1$. 
In our situation, a point $R$ with $|G[R]|=10m$ is unique, by Theorem \ref{outer}(1). 
This implies that ${\rm Aut}(C)$ fixes $P_1''$. 
Let $R \ne P_1''$ be a point with $|G[R]|=5$. 
Then $F[R] \ni P''$. 
By Lemma \ref{pair 1}, $(R, P'')$ is a $G$-pair. 
Since $F[P''] \setminus \{P''\}=\ell$, it follows that $R \in \ell$. 
The proof of assertion (2) is completed. 

Let $n=4$.  
Assume that there does not exist a $G$-pair on the line $\ell$. 
Then there exist at least $5$ cyclic subgroups of $\varphi(G) \cong S_4$ of order four. 
This is a contradiction, since $S_4$ has exactly three cyclic subgroups of order four. 
Therefore, there exists a $G$-pair $(P, P')$ on the line $\ell$. 
Similarly to the previous paragraph, there exists a point $P_1' \in \ell$ such that $(P_1, P_1')$ is a $G$-pair. 
If $\#\Delta_4' \cap \ell \ge 3$, then $\varphi(G) \cong S_4$. 
Then, by the action of $G[P_1]$, we have $6$ such points on $\ell$. 
Since $S_4$ has exactly three cyclic subgroups of order four, we have exactly $6$ quasi-Galois points on this line. 
Note that, by Proposition \ref{pair 2}, there exists a point $P_1'' \not\in \ell$ such that $(P_1, P_1')$, $(P_1', P_1'')$ and $(P_1'', P_1)$ are $G$-pairs. 

Assume that $P_2 \in \ell$ is a quasi-Galois point with $|G[P_2]|=4$ and $P_2 \ne P_1, P_1'$,  and that $R \not \in \ell$ is a quasi-Galois point with $|G[R]|=4$. 
If $R \not \in \overline{P_1'P_1''}$, then there exists a quasi-Galois point $R' \in \overline{P_1R}$ such that $(P_1, R')$ is a $G$-pair, and hence, $R'$ must be in $\overline{P_1'P_1''}$. 
Therefore, we can assume that $R \in \overline{P_1'P_1''}$ with $R \ne P_1', P_1''$. 
Let $\eta \in G[P_2]$ be the involution. 
By Lemma \ref{n=4} below, $\eta(P_1)=P_1'$ and $\eta(P_1')=P_1$. 
Since $\eta(P_1'')=P_1''$, it follows that $\eta(R) \in \overline{P_1''P_1}$. 
By Lemma \ref{n=4} again, $(R, \eta(R))$ is a $G$-pair. 
Since $P_1, \eta(R) \in F[R]\setminus \{R\}$, it follows that $F[R] \setminus \{R\}=\overline{P_1\eta(R)}=\overline{P_1''P_1}=F[P_1'] \setminus \{P_1'\}$. 
By Proposition \ref{two quasi-Galois}(2), this is a contradiction. 
Therefore, $\Delta_4' \subset (\Delta_4' \cap \ell) \cup \{P_1''\}$.     
\end{proof} 

\begin{lemma} \label{n=5}
Let $\ell$ be a line containing $12$ quasi-Galois points $P \in \mathbb{P}^2$ with $|G[P]|=5$, let $P, P' \in \ell$ form a $G$-pair, and let $P'' \in F[P] \cap F[P']$. 
If $\varphi(G) \cong A_5$, then $|G[P'']|=10m$ for some integer $m \ge 1$. 
\end{lemma}

\begin{proof} 
We can assume that points $P=(1:0:0)$, $P'=(0:1:0)$ form a $G$-pair with $P, P' \in \ell$ and $|G[P]|=|G[P']|=5$, and $\sigma \in G[P]$, $\sigma' \in G[P']$ are generators represented by 
$$A_{\sigma}= 
\left(\begin{array}{ccc}
\zeta & 0 & 0 \\
0 & 1 & 0 \\
0 & 0 & 1
\end{array}\right), \ 
A_{\sigma'}=
\left(\begin{array}{ccc}
1 & 0 & 0 \\
0 & \zeta & 0 \\
0 & 0 & 1 
\end{array}\right), 
$$
where $\zeta$ is a primitive $5$-th roof of unity. 
Further, we can assume that $P''=(0:0:1) \in F[P] \cap F[P']$. 
Let $\overline{\sigma}=\varphi(\sigma)$. 
Since $\varphi(G) \cong A_5$, it follows that there exists an involution $\overline{\tau} \in \varphi(G)$ such that $\overline{\sigma}(\overline{\tau} \overline{\sigma} \overline{\tau})\overline{\sigma}=\overline{\tau}$.  
We consider $\overline{\tau}$ as an element of ${\rm PGL}(2, K)$. 
Since $\overline{\tau}$ is an involution and $\overline{\tau}$ does not fix $(1:0)$ or $(0:1)$, it follows that $\overline{\tau}$ is represented by the matrix 
$$ 
A_{\overline{\tau}}=\left(\begin{array}{cc}
1 & b \\
c & -1 
\end{array}\right)
$$
for some $b, c \in K$. 
Let 
$$ B=\left(\begin{array}{cc}
\sqrt{\frac{b}{c}} & 0 \\ 
0  & 1 
\end{array}\right). 
$$
Then
$$ B^{-1}A_{\overline{\tau}}B=
\left(\begin{array}{cc}
\sqrt{\frac{b}{c}} & b \\
b & -\sqrt{\frac{b}{c}} 
\end{array}\right). 
$$
Therefore, we can assume that 
$$ 
A_{\overline{\tau}}=\left(\begin{array}{cc}
1 & \alpha \\
\alpha & -1 
\end{array}\right)
$$
for some $\alpha \in K$. 
It follows that 
$$ A_{\overline{\tau}} A_{\overline{\sigma}} A_{\overline{\tau}}
=\left(\begin{array}{cc}
\zeta+\alpha^2 & (\zeta-1)\alpha \\
(\zeta-1)\alpha & \zeta\alpha^2+1 
\end{array}\right), \
A_{\overline{\sigma}} (A_{\overline{\tau}} A_{\overline{\sigma}}A_{\overline{\tau}}) A_{\overline{\sigma}}=
\left(\begin{array}{cc}
\zeta^2(\zeta+\alpha^2) & \zeta(\zeta-1)\alpha \\
\zeta(\zeta-1)\alpha & \zeta\alpha^2+1
\end{array}\right). 
$$
Since $\overline{\sigma}(\overline{\tau} \overline{\sigma} \overline{\tau})\overline{\sigma}=\overline{\tau}$, it follows that 
$$ \alpha^2=1-\zeta-\zeta^4. $$
The fixed locus of the linear transformation $\overline{\tau} \overline{\sigma} \overline{\tau}$ consists of two points 
$$ (1:\alpha), \ (\alpha:-1). $$
Since $\overline{\tau}\overline{\sigma}\overline{\tau}$ is of order five and is contained in $\varphi(G)$, it follows that $P_2:=(1:\alpha:0)$ is a quasi-Galois point with $|G[P_2]|=5$ and there exists a generator $\sigma_2 \in G[P_2]$ such that $\varphi(\sigma_2)=\overline{\tau}\overline{\sigma}\overline{\tau}$. 
Let $P_2':=(\alpha:-1:0)$. 
Then $F[P_2]=\{P_2\} \cup \{X+\alpha Y=0\}$. 
It is inferred that $\sigma_2$ is represented by the matrix
$$
A_{\sigma_2}=\left(\begin{array}{ccc} 
\zeta+\alpha^2 & (\zeta-1)\alpha  & 0 \\
(\zeta-1)\alpha & \zeta\alpha^2+1 & 0 \\
0 & 0 & \alpha^2+1 
\end{array}\right). 
$$  
Then 
$$ 
A_{\sigma}A_{\sigma_2}A_{\sigma}=
\left(\begin{array}{ccc} 
\zeta^3+\zeta^2\alpha^2 & (\zeta^2-\zeta)\alpha  & 0 \\
(\zeta^2-\zeta)\alpha & \zeta\alpha^2+1 & 0 \\
0 & 0 & \alpha^2+1 
\end{array}\right). 
$$
Note that $\zeta^3+\zeta^2 \alpha^2=\zeta(\zeta-1)$ and $\zeta \alpha^2+1=\zeta(1-\zeta)$. 
It follows that 
$$(A_{\sigma}A_{\sigma_2}A_{\sigma})^2 \sim
\left(\begin{array}{ccc}
\zeta^2(\zeta-1)^2 & 0 & 0 \\
0 & \zeta^2(\zeta-1)^2& 0 \\ 
0 & 0 & \alpha^2+1  
\end{array}\right) \in G[P'']. $$ 
It can be confirmed that 
$$\frac{\alpha^2+1}{\zeta^2(\zeta-1)^2}=-\zeta^2. $$
Since $(\sigma \sigma_2 \sigma)^2$ is of order $10$, it follows that $|G[P'']|=10m$ for some $m \ge 1$.  
\end{proof}

\begin{lemma} \label{n=4}
Let $|G[P]|=4$ and let $\eta \in G[P]$ be an involution. 
If $|G[R]|=4$ and $\eta(R) \ne R$, then $(R, \eta(R))$ is a $G$-pair. 
\end{lemma}

\begin{proof}
Since $R \not\in F[P]$, $\varphi(\langle G[P], G[R] \rangle) \cong S_4$. 
Note that $S_4$ has exactly three cyclic subgroup of order four, and $\varphi(G[P])$ acts on the set of subgroups of order four different from $\varphi(G[P])$. 
This implies that $\varphi(\eta)$ fixes each cyclic subgroup of order four. 
Then $\varphi(G[\eta(R)])=\varphi(\eta G[R] \eta^{-1})=\varphi(G[R])$. 
This implies that $F[R] \ni \eta(R)$. 
By Lemma \ref{pair 1}, $(R, \eta(R))$ is a $G$-pair. 
\end{proof}

\begin{corollary}
Assume that $d \ge 8$, $d$ is even, and $n=d/2$.   
Then, $\delta'[\ge n] \ge 2$ if and only if $C$ is projectively equivalent to the curve defined by  
$$ X^{2n}+Y^{2n}+Z^{2n}+aX^{n}Y^{n}+bY^nZ^n+cZ^nX^n=0, $$
where $a, b, c \in K$. 
In this case, if $d \ge 10$ (resp., $d=8$), then $\delta'[\ge n]=3$ (resp., $\delta'[\ge 4]=3$ or $7$). 
\end{corollary}

\begin{proof}
The former assertion is derived from Corollary \ref{pair 2'} and Theorem \ref{outer}. 
The latter assertion for the case where $d \ge 12$ or $d=8$ is obvious, by Theorem \ref{outer}. 
Assume that $d=10$. 
Let $P \in \mathbb{P}^2 \setminus C$ be a point with $|G[P]| \ge 5$. 
By Proposition \ref{tangent1}, there exist $d=10$ points $Q \in C \cap (F[P] \setminus \{P\})$ such that $P \in T_QC$ and $I_Q(C, T_QC) \ge 5$. 
Therefore, for each quasi-Galois point $P \in C$ with $|G[P]|\ge 5$, we need at least $10 \times (5-2)=30$ flexes with multiplicities. 
It follows from Proposition \ref{two quasi-Galois} that there exists no point $Q \in C$ such that $Q \in F[P_1] \cap F[P_2]$ for different quasi-Galois points $P_1$ and $P_2$ with $|G[P_1]| \ge 5$ and $|G[P_2]| \ge 5$. 
By the flex formula \cite[Theorem 1.5.10]{namba}, we have $\delta'[\ge 5] \times 30 \le 3d(d-2)=240$. 
Therefore, $\delta'[\ge 5] \le 8$. 
By Theorem \ref{outer}, it follows that $\delta'[\ge 5]=3$.  
\end{proof}

\begin{theorem} \label{outer n=3}
Assume that $n=3$, and points $P_1$ and $P_2 \in \mathbb{P}^2 \setminus C$ are different quasi-Galois points with $|G[P_1]|=|G[P_2]|=3$.
Let $\ell=\overline{P_1P_2}$. 
Then the following hold. 
\begin{itemize}
\item[(1)] $\#\Delta'_3 \cap \ell=2, 4, 8$ or $20$.  
Furthermore, if $\#\Delta_3' \cap \ell=8$ or $20$, then there exists $P_1' \in \ell$ such that $(P_1, P_1')$ is a $G$-pair. 
\item[(2)] If $\#\Delta'_3 \cap \ell=8$, then $\delta'[3]=8$ and there exists a unique integer $m \ge 1$ such that $\delta'[6m]=1$. 
\item[(3)] If $\#\Delta'_3 \cap \ell=20$, then $\delta'[3]=20$ and there exists a unique integer $m \ge 1$ such that $\delta'[6m]=1$. 
\item[(4)] If $\#\Delta'_3 \cap \ell=4$, then $\delta'[3]=4$ or $12$. 
\end{itemize} 
In particular, $\delta'[3]=2, 3, 4, 8, 12$ or $20$. 
\end{theorem}

\begin{proof} 
Assume that there does not exist a $G$-pair on the line $\ell$. 
Then there exist $3m+1$ quasi-Galois points for some integer $m$ by the actions associated with one quasi-Galois point.
Then, by Fact \ref{PGL}, 
$$ 3(3m+1)=12, 24 \mbox{ or } 60. $$
This implies that $m=1$ and $\#\Delta_3' \cap \ell=4$. 

We assume that $\#\Delta_3' \cap \ell \ge 5$. 
Then there exists a $G$-pair $(P, P')$ on the line $\ell$. 
By Proposition \ref{pair 2}, there exists a quasi-Galois point $P''$ with $F[P''] \setminus \{P''\}=\ell$. 
Since $P_1 \in \ell=F[P''] \setminus \{P''\}$, it follows from Lemma \ref{pair 1} that $(P_1, P'')$ is a $G$-pair. 
By Proposition \ref{pair 2}, there exists a quasi-Galois point $P_1' \in \ell \cap F[P_1]$. 
Then $(P_1, P_1')$ is a $G$-pair. 

By the discussion in the previous paragraph, $\#\Delta_3' \cap \ell$ is even, and hence, $\#\Delta_3' \cap \ell \ge 8$. 
By Fact \ref{PGL}, $\varphi(G)=A_4, S_4$ or $A_5$. 
Since $\varphi(G[P])$ is a Sylow $3$-group of $\varphi(G)$, each subgroup of order three is realized as $\varphi(G[P])$ ($=\varphi(G[P'])$) for some exactly two quasi-Galois points $P, P' \in \Delta_3' \cap \ell$. 
If $\varphi(G) \cong A_4$ or $S_4$ (resp., $\varphi(G) \cong A_5$), then the number of subgroup of order three is $4$ (resp., $10$). 
Therefore, the number of quasi-Galois points on $\ell$ is $8$ or $20$. 
Assertion (1) follows.   

Assume that $\#\Delta'_3 \cap \ell=8$. 
Then $\varphi(G) \cong A_4$ or $S_4$. 
In this case, for points $P_1$ and $P_2$ such that $(P_1, P_2)$ is not a $G$-pair, $\varphi(\langle G[P_1], G[P_2] \rangle \cong A_4$. 
Then the obit $A_4P_1$ has length three. 
Note that $(P_1, P_1')$, $(P_1', P_1'')$ and $(P_1'', P_1)$ are $G$-pairs. 
By Lemma \ref{four quasi-Galois}, there exists a point $Q \not\in \ell$ such that $|G[Q]|$ is even and $F[Q] \setminus \{Q\}=\ell$. 
Then $Q=P_1''$, because the intersection point of tangent lines at $d$ points of $C$ on the line $\ell$ is unique. 
Since $G[P_1'']$ contains elements of order three and two, the order $|G[P_1'']|$ is equal to $6m$ for some $m$. 
In our situation, a point $R$ with $|G[R]|=6m$ is unique, by Theorem \ref{outer}(1). 
This implies that ${\rm Aut}(C)$ fixes $P_1''$. 
Let $R \ne P_1''$ be a point with $|G[R]|=3$. 
Then $F[R] \ni P_1''$. 
By Lemma \ref{pair 1}, $(R, P_1'')$ is a $G$-pair. 
Since $F[P_1''] \setminus \{P_1''\}=\ell$, it follows that $R \in \ell$. 
Assertion (2) follows. 

Assume that $\#\Delta_3' \cap \ell=20$. 
Then $\varphi(G) \cong A_5$. 
Note that all subgroups of $A_5$ of order three are realized as the image of associated Galois groups of quasi-Galois points under the restriction $\varphi$. 
This implies that there exists a pair of points $P_1, P_2 \in \Delta_3' \cap \ell$ such that $(P_1, P_2)$ is not a $G$-pair and $\varphi(\langle G[P_1], G[P_2] \rangle) \cong A_4$. 
The same argument as assertion (2) can be applied to assertion (3).  

We consider assertion (4). 
Assume that $\#\Delta'_3 \cap \ell=4$. 
According to assertions (1), (2) and (3), we can assume that for all lines $\ell' \subset \mathbb{P}^2$, $\#\Delta'_3 \cap \ell'=0$, $1$, $2$ or $4$. 
By Lemma \ref{four quasi-Galois}, there exists a quasi-Galois point $Q \not\in \ell$ such that $|G[Q]|$ is even and $F[Q] \setminus \{Q\}=\ell$. 
Let $\tau \in G[Q]$ be an involution. 
We prove that there does not exist a line $\ell' \ni Q$ with $\#\Delta'_3 \cap \ell'=4$. 
Assume by contradiction that $\#\Delta'_3 \cap \ell'=4$ and $\Delta'_3 \cap \ell'=\{P_1', P_2', P_3', P_4'\}$. 
Note that $\ell' \cap \ell \cap \Delta'_3=\emptyset$, by considering the action of $\tau$. 
Let $Q' \not\in \ell'$ be a quasi-Galois point such that $|G[Q']|$ is even and $F[Q'] \setminus \{Q'\}=\ell'$. 
Since the point $Q'$ is contained in the tangent line for any point in $C \cap \ell'$ and $G[Q]$ acts on $\ell'$, it follows that $Q' \in F[Q]$, namely, $Q' \in \ell$. 
Note that $G[Q']$ acts on the set $\Delta'_3 \cap \ell$. 
The group $G[Q']$ does not fix any point in $\Delta'_3 \cap \ell$, since $Q' \not\in \Delta'_3$ and $\ell' \cap \ell \cap \Delta'_3=\emptyset$. 
Let $\tau' \in G[Q']$ be an involution. 
Then there exist a quasi-Galois point $P \in \Delta'_3 \cap \ell$ and an automorphism $\sigma \in G[P]$ such that three points $Q', \sigma(Q'), \sigma^2(Q')$ are different. 
Let $Q'_2=\sigma(Q'), Q_3'=\sigma^2(Q')$. 
Then $\tau_2':=\sigma\tau'\sigma^{-1} \in G[Q'_2]$ and $\tau_2':=\sigma^2\tau'\sigma^{-2} \in G[Q'_3]$.  
Since $\tau'$ and $\sigma$ act on $\Delta'_3 \cap \ell$, it follows that $\tau'|_\ell$, $(\sigma\tau'\sigma^{-1})|_\ell$, $(\sigma^{2}\tau\sigma^{-2})|_\ell$ are different involutions on $\ell$. 
Since the number of involutions acting on four points not fixing any point of them is at most three, it follows that $(\tau_2'\tau_3') |_\ell=\tau' |_\ell$.  
Note that $(Q_2', Q_3')$ is not a $G$-pair, since if $(Q_2', Q_3')$ is a $G$-pair, then $\tau_2' |_\ell=\tau_3' |_\ell$. 
It follows from Lemma \ref{pair 1} that $\tau_3'(Q_2') \ne Q_2'$. 
If $\tau_2'\tau_3'$ is an involution as an automorphism of $\mathbb{P}^2$, then $\tau_3'(Q_2')$ is a quasi-Galois point with $\tau_3'\tau_2'\tau_3'=\tau_2' \in G[\tau_3'(Q_2')] \cap G[Q_2']$. 
By Corollary \ref{two groups}, this is a contradiction. 
Therefore, the order of $\tau_2'\tau_3'$ is at least $3$. 
Since $(\tau_2'\tau_3') |_\ell=\tau' |_\ell$, it follows that $(\tau_2'\tau_3')(\ell')=\ell'$. 
It follows that $\tau_2'\tau_3'$ acts on $\Delta'_3 \cap \ell'$ faithfully, namely, $\tau_2'\tau_3' |_{\ell'}$ is of order four. 
Let $G' \subset {\rm Aut}(\ell')$ be the group arising from the restrictions of all automorphisms in $\langle G[P_1'], G[P_2'], \tau_2'\tau_3' \rangle$ on the line $\ell'$. 
Then $G' \cong S_4$.   
As the group $S_4$, the stabilizer subgroup of $P_1'$ is $S_3$.
As a finite subgroup of ${\rm Aut}(\ell')$, the stabilizer subgroup of $P_1'$ is a cyclic group.  
This is a contradiction. 

Assume that there exists a point $P \in \Delta'_3$ with $P \not\in \ell=\overline{P_1P_2}$. 
Then there exists a point $P' \in \Delta'_3 \cap \overline{QP}$ with $P' \not\in \ell \cup \{P\}$, according to the action of an involution in $G[Q]$. 
Since the group $\langle G[P_1], G[P_2] \rangle$ acts on the set of all lines passing through $Q$, it follows that $\delta'[3] \ge 12$.   
If $(P, P')$ is not a $G$-pair, then $\#\Delta'_3 \cap \overline{QP}=4$. 
According to the above discussion, this is a contradiction. 
Therefore, $(P, P')$ is a $G$-pair. 
There exists a point $P'' \not \in \overline{QP}$ such that $(P, P'')$ and $(P', P'')$ are $G$-pairs. 
Since $P'' \in F[P] \cap F[P'] \subset F[Q] \setminus \{Q\}=\ell$, it follows that $P'' \in \Delta'_3 \cap \ell$ or $|G[P'']| \ge 6$. 
For the latter case, by the action of $G[P'']$, there exist at least six points $P''' \in \Delta'_3 \cap \ell$ with $|G[P''']|=3$. 
This is a contradiction. 
Threfore, $P'' \in \Delta'_3 \cap \ell$ holds. 
This implies that $\delta'[3]=12$. 
\end{proof}

\section{Curves of degree six}

We consider the case where $d=6$ and $n=3$. 
We determine the number $\delta'[3]$. 

\begin{theorem} \label{sextic} 
Let $C \subset \mathbb{P}^2$ be a smooth plane curve of degree $d=6$. 
Then $$ \delta'[3]=0, 1, 2, 3, 4, 8 \mbox{ or } 12. $$ 
Furthermore, the following hold. 
\begin{itemize}
\item[(1)] $\delta'[3]=12$ if and only if $C$ is projectively equivalent to the curve defined by 
$$ X^6+Y^6+Z^6-10(X^3Y^3+Y^3Z^3+Z^3X^3)=0. $$
\item[(2)] $\delta'[3]=8$ if and only if $C$ is projectively equivalent to the curve defined by 
$$ X^6+20X^3Y^3-8Y^6+Z^6=0. $$
\item[(3)] $\delta'[3]=4$ if and only if $C$ is projectively equivalent to the curve defined by 
$$ Z^6+a(X^3Y+Y^4)Z^2+(X^6+20X^3Y^3-8Y^6)=0 $$
for some $a \in K \setminus \{0\}$, and $C$ is not in the case (1).  
\end{itemize} 
\end{theorem}

\begin{proof} 
Let $P \in \mathbb{P}^2 \setminus C$ be a point with $|G[P]|=3$. 
By Proposition \ref{tangent1}, there exist $d=6$ points $Q \in C \cap (F[P] \setminus \{P\})$ such that $P \in T_QC$ and $I_Q(C, T_QC) \ge 3$. 
Therefore, for each quasi-Galois point $P \in C$ with $|G[P]|=3$, we need at least $6$ flexes. 
It follows from Proposition \ref{two quasi-Galois} that there exists no point $Q \in C$ such that $Q \in F[P_1] \cap F[P_2]$ for different quasi-Galois points $P_1$ and $P_2$ with $|G[P_1]|=|G[P_2]|=3$. 
By the flex formula \cite[Theorem 1.5.10]{namba}, we have $\delta'[3] \times 6 \le 72$. 
Therefore, $\delta'[3] \le 12$.  

Assume that $\delta'[3] \ge 5$. 
First, we prove that there exists a $G$-pair. 
Assume by contradiction that $\sigma(P_2) \ne P_2$ for any quasi-Galois points $P_1, P_2$ and any generator $\sigma \in G[P_1]$. 
By Theorem \ref{outer n=3}, quasi-Galois points are not collinear. 
Then it is inferred that there exist $3^2+3+1=13$ quasi-Galois points. 
This is a contradiction. 
Therefore, there exists a $G$-pair $(P, P')$. 
According to Proposition \ref{pair 2}, for a suitable system of coordinates, we can assume that $P=(1:0:0)$, $P'=(0:1:0)$, generators $\sigma$ of $G[P]$ and $\sigma'$ of $G[P']$ are given by the matrices 
$$ A_{\sigma}=\left(\begin{array}{ccc} 
\omega & 0 & 0 \\
0 & 1 & 0 \\
0 & 0 & 1 \end{array}\right), \ 
A_{\sigma'}=\left(\begin{array}{ccc} 
1 & 0 & 0 \\
0 & \omega & 0 \\
0 & 0 & 1 \end{array}\right) $$
respectively, where $\omega^2+\omega+1=0$, and $C$ is given by 
$$X^6+aY^6+bZ^6+cX^3Y^3+dY^3Z^3+eZ^3X^3=0, $$
where $a, b, c, d, e \in K$. 
Then $P''=(0:0:1)$ is also quasi-Galois. 

Next, we consider the case where there exists a quasi-Galois point not contained in the set $S:=\overline{PP'}\cup\overline{P' P''}\cup\overline{P'' P}=\{XYZ=0\}$. 
Let $(\alpha :\beta:1)$ be a quasi-Galois point on $\mathbb{P}^2 \setminus S$. 
By using the linear transformation given by $(X:Y:Z) \mapsto ((1/\alpha)X:(1/\beta)Y:Z)$, we can assume that $(\alpha:\beta:1)=(1:1:1)$. 
Then points 
$$ P_{ij}:=(\omega^i:\omega^j:1)$$
are quasi-Galois for $i, j=0, 1, 2$, and the set $\{P, P', P''\} \cup \{P_{ij} \ | \ i, j=0, 1, 2\}$ consists of all quasi-Galois points for $C$. 

We compute a generator $\tau \in G[P_{00}]$, where $P_{00}=(1:1:1)$. 
It follows that $\tau(P)=P_{i0}$, $\tau(P')=P_{0j}$ and $\tau(P'')=P_{kk}$ for some $i, j, k \ne 0$. 
We can assume that $k=2$ and $\tau(P'')=(1:1:\omega)$.
Then $\tau$ is represented by the matrix
$$ A_{\tau}=\left( 
\begin{array}{ccc}
\lambda \omega^i & \mu & 1 \\ 
\lambda & \mu \omega^j & 1 \\
\lambda & \mu & \omega 
\end{array}
\right), $$ 
for some $\lambda, \mu \in K \setminus \{0\}$. 
By using the condition $\tau((1:1:1))=(1:1:1)$, it follows that
$$ \lambda=\frac{\omega-1}{\omega^i-1}, \ \mu=\frac{\omega-1}{\omega^j-1}. $$
If $i=2$, then $\lambda=1/(\omega+1)=-\omega$ and $\tau((1:0:1))=(0:1:0)$. 
Since $(1:0:1)$ is not quasi-Galois, this is a contradiction. 
It follows that $i=1$ and $\lambda=1$. 
Similarly, it follows that $j=1$ and $\mu=1$. 

Note that 
$$ A_{\sigma_1}A_{\tau}A_{\sigma_1}A_{\tau}=
\left(\begin{array}{ccc} 
3\omega & 0 & 0 \\
0 & 0 & 3\omega \\
0 & 3\omega & 0 
\end{array}\right). $$
The linear transformation given by $(X:Y:Z) \mapsto (X:Z:Y)$ acts on $C$. 
Similarly, the linear transformation given by $(X:Y:Z) \mapsto (Z:Y:X)$ acts on $C$. 
Therefore, the defining equation of $C$ is of the form 
$$ F=X^6+Y^6+Z^6+a(X^3Z^3+Y^3Z^3+Z^3X^3)=0$$
for some $a \in K$. 
We consider the action by $\tau$. 
Polynomials $(\tau^{-1})^*F$ and $F$ are the same up to a constant. 
We consider the coefficient of $X^4YZ$. 
It follows that the coefficient of $X^4YZ$ is $30\omega$ for $(\omega X+Y+Z)^6, (X+\omega Y+Z)^6$ and $(X+Y+\omega Z)^6$. 
The coefficient is $3\omega$ for $(\omega X+Y+Z)^3(X+\omega Y+Z)^3$, $(X+\omega Y+Z)^3(X+Y+\omega Z)^3$ and $(X+Y+\omega Z)^3(\omega X+Y+Z)^3$. 
It follows that $a=-10$. 

We consider the case where all quasi-Galois points are contained in the set $S$.  
If there exist two quasi-Galois points $P_2, P_3 \not\in \{P, P', P''\}$ which are contained in $X=0$ and $Y=0$ respectively, then $(P_2, P_3)$ is not a $G$-pair, since $(P_2, P)$ is a $G$-pair, $(P_2, P')$ is not a $G$-pair, and $P_3 \in \overline{PP''}$. 
We can find quasi-Galois points in $\mathbb{P}^2\setminus S$ by the actions associated with $P_2$. 
Therefore, we can assume that quasi-Galois points different from $\ne P, P', P''$ are contained in one line $\subset S$.  
We can assume that such a line is $\overline{P P'}$. 
Then, by Theorem \ref{outer n=3}, $\#\Delta_3' \cap \overline{P P'}=8$.  
Furthermore, $|G[P'']|=6$, that is, $P''$ is a Galois point.  
Let $P_1=P$. 
We use the same symbols in the proof of Lemma \ref{four quasi-Galois}.
It follows that $\sigma_2\sigma_1^2$ is represented by the matrix
$$ A_{\sigma_2\sigma_1^2}=\left(\begin{array}{ccc}
-\alpha & 2\omega^2 \alpha & 0 \\
\omega \alpha& \alpha & 0 \\ 
0 & 0 & 1
\end{array}\right). $$
Note that the restriction of $\sigma_2\sigma_1^2$ on the line $\overline{P P'}$ is of order two, and fixed points of $\sigma_2\sigma_1^2$ on the line $\overline{P P'}$ are $((-1+\sqrt{3})\omega^2:1:0)$ and $((-1-\sqrt{3})\omega^2:1:0)$. 
Since $A_4$ acts on the six points given by $C \cap \overline{P P'}$ and contains exactly three elements of order two, it follows that 
$$C \cap \overline{P P'}=\{(-1+\sqrt{3})\omega^i:1:0) \ | \ i=0, 1, 2\} \cup \{((-1-\sqrt{3})\omega^i:1:0) \ | \ i=0, 1,2\}.$$ 
Then the defining equation of $C$ is of the form 
$$ F(X, Y, Z)=X^6+20X^3Y^3-8Y^6+Z^6=0. $$ 

Assume that $\delta'[3]=4$. 
Then the four quasi-Galois points $P_1, P_2, P_3, P_4$ are contained in a unique line $\ell$, and $\varphi(G) \cong A_4$. 
We can assume that $P_1=(1:0:0)$, $P_2=(1:-1:0)$, $\sigma_1 \in G[P_1]$ and $\sigma_2 \in G[P_2]$ are generators represented by 
$$ 
A_{\sigma_1}=\left(\begin{array}{ccc}
\omega & 0 & 1 \\
0 & 1 & 0 \\
0 & 0 & 1 
\end{array}\right), \
A_{\sigma_2}=\left(\begin{array}{ccc}
-\omega \alpha & 2\omega^2 \alpha & 0 \\
\omega^2 \alpha & \alpha & 0 \\
0 & 0 & 1 
\end{array}\right)
$$
respectively. 
By Lemma \ref{four quasi-Galois}, it follows that $|G[Q]| \ge 2$ and $\ell=F[Q] \setminus \{Q\}$ for the point $Q=(0:0:1) \in F[P_1] \cap F[P_2]$.   
Then the defining equation of $C$ is of the form  
$$
F(X, Y, Z)=Z^6+a Y^2Z^4+(b X^3Y+c Y^4)Z^2+G(X, Y)=0, 
$$
where $a, b, c \in K$ and $G$ is a homogeneous polynomial of degree $6$. 
Similarly to the previous paragraph, we can assume that 
$$ G(X, Y)=X^6+20X^3Y^3-8Y^3. $$
By comparing the coefficient of $Y^2Z^4$ of $F(-\omega \alpha X+2\omega^2 \alpha Y, \omega^2\alpha X+\alpha Y, Z)$ and $F(X, Y, Z)$, it follows that $a=0$.  
By comparing the coefficient of $X^4Z^2$ of $F(-\omega \alpha X+2\omega^2 \alpha Y, \omega^2\alpha X+\alpha Y, Z)$ and $F(X, Y, Z)$, it follows that $b=c$. 
Then the defining equation of $C$ is of the form 
$$ 
Z^6+c(X^3Y+Y^4)Z^2+(X^6+20X^3Y^3-8Y^6)=0. 
$$

On the contrary, we consider the curve $C$ given by 
$$ X^6+Y^6+Z^6-10(X^3Y^3+Y^3Z^3+Z^3X^3)=0. $$
By Fact \ref{standard form}, points $P=(1:0:0)$, $P'=(0:1:0)$ are quasi-Galois, and groups $G[P], G[P']$ are generated by the linear transformations $\sigma, \sigma'$ given by 
$$ A_{\sigma}=\left(\begin{array}{ccc} 
\omega & 0 & 0 \\
0 & 1 & 0 \\
0 & 0 & 1 \end{array}\right), \ 
A_{\sigma'}=\left(\begin{array}{ccc} 
1 & 0 & 0 \\
0 & \omega & 0 \\
0 & 0 & 1 \end{array}\right) $$
respectively. 
Let $\tau$ be the linear transformation given by the matrix 
$$ A_{\tau}=\left(\begin{array}{ccc}
\omega & 1 & 1 \\
1 & \omega & 1 \\
1 & 1 & \omega 
\end{array} \right). $$
Then $\tau(C)=C$ and 
$$ \tau^*\left(\frac{x-y}{y-1}\right)=\frac{x-y}{y-1}.  $$ 
Therefore, the point $(1:1:1)$ is quasi-Galois on $\mathbb{P}^2 \setminus \{XYZ=0\}$. 
By considering the actions of $\sigma$ and $\sigma'$, it follows that $\delta'[3] \ge 12$. 
Since we confirmed $\delta'[3] \le 12$ in the first paragraph, it follows that $\delta'[3]=12$.

We consider the curve $C$ defined by 
$$ F(X, Y, Z)=X^6+20X^3Y^3-8Y^6+Z^6=0. $$
To prove $\delta'[3]=8$, we have to prove that the linear transformation $\sigma_2$ represented by 
$$A_{\sigma_2}=\left(\begin{array}{ccc}
-\omega \alpha & 2\omega^2 \alpha & 0 \\
\omega^2 \alpha & \alpha & 0 \\
0 & 0 & 1
\end{array}\right) $$
acts on $F$. 
To do this, we prove that the linear transformation $\sigma_2\sigma_1^2$ represented by 
$$ A_{\sigma_2\sigma_1^2}=\left(\begin{array}{ccc}
-\alpha & 2\omega^2 \alpha & 0 \\
\omega \alpha & \alpha & 0 \\ 
0 & 0 & 1
\end{array}\right) $$
acts on $C$. 
Let $G(X, Y)=X^6+20X^3Y^3-8Y^6$. 
It is easily verified that the coefficient of $X^6$ for $G(-\alpha X+2\omega^2 \alpha Y, \omega \alpha X+\alpha Y)$ is $1$. 
Further, it is inferred that the set  
$$\{(-1+\sqrt{3})\omega^i:1:0) \ | \ i=0, 1, 2\} \cup \{((-1-\sqrt{3})\omega^i:1:0) \ | \ i=0, 1,2\}$$ 
is invariant under the action of $A_{\sigma_2\sigma_1^2}$. 
The claim follows. 

We consider the curve defined by 
$$ F(X, Y, Z)=Z^6+a(X^3Y+Y^4)Z^2+(X^6+20X^3Y^3-8Y^6)=0, $$
where $a \in K \setminus \{0\}$.  
Assume that $C$ is not in the case (1). 
It is obvious that $F(\omega X, Y, Z)=F$. 
According to the previous paragraph, for the polynomial $G(X, Y)=X^6+20X^3Y^3-8Y^6$, 
$$G(-\omega\alpha X+2\omega^2 \alpha Y, \omega^2 \alpha X+\alpha Y)=G(X, Y). $$
It is not difficult to confirm that for the polynomial $H(X, Y)=X^3Y+Y^4$, 
$$H(-\omega\alpha X+2\omega^2 \alpha Y, \omega^2 \alpha X+\alpha Y)=H(X, Y) $$
(see also \cite[Lemma 1]{kty}). 
Therefore, the linear transformation represented by 
$$A_{\sigma_2}=\left(\begin{array}{ccc}
-\omega \alpha & 2\omega^2 \alpha & 0 \\
\omega^2 \alpha & \alpha & 0 \\
0 & 0 & 1
\end{array}\right) $$
acts on $C$, and hence, $\delta'[3]=4$ or $8$. 
If $\delta'[3]=8$, then the point $(0:0:1)$ must be a Galois point. 
This forces $a=0$. 
The claim $\delta'[3]=4$ follows. 
\end{proof} 

\begin{remark}
If $\delta'[3]=2$, then $\delta'[6]=1$. 
If $\delta'[3]=4$, then $\delta'[2]=1$. 
\end{remark}

As an application, on the automorphism group ${\rm Aut}(C)$, we have the following (see Part I \cite{fmt} for the definition of $G_3(C)$). 

\begin{theorem} \label{sextic, generated}
Let $C$ be the plane curve defined by $X^6+Y^6+Z^6-10(X^3Y^3+Y^3Z^3+Z^3X^3)=0$. 
Then
$$ G_3(C)={\rm Aut}(C). $$
\end{theorem}

\begin{proof}
Let $P=(1:0:0)$, $P'=(0:1:0)$, $P''=(0:0:1)$, and let $P_{ij}=(\omega^i:\omega^j:1)$ for $i, j=0, 1, 2$, where $\omega^2+\omega+1=0$. 
Then the set $\Delta':=\{P, P', P''\} \cup \{P_{ij} \ | \ i, j=0, 1, 2\}$ consists of all quasi-Galois points $P$ with $|G[P]|=3$. 

Let $\sigma \in {\rm Aut}(C) \subset {\rm Aut}(\mathbb{P}^2)$.
Then $\sigma$ acts on $\Delta'$. 
If $\sigma(P)=P'$ or $P''$, then there exists $\phi_1 \in G_3(C)$ such that $\phi_1\sigma(P)=P$,  since the automorphisms $(X:Y:Z) \mapsto (Y:X:Z)$ and $(X:Y:Z) \mapsto (Z:Y:X)$ are contained in $G_3(C)$ as in the proof of Theorem \ref{sextic}. 
If $\sigma(P)=P_{ij}$ for some $i, j$, then there exists $\phi_2 \in G[P_{kj}]$ for $k \ne i$ such that $\phi_2\sigma(P)=P$. 
Therefore, there exists $\phi \in G_3(C)$ such that $\phi\sigma(P)=P$. 
The line $F[P]\setminus \{P\}$ is a unique line $\ell$ such that $I_Q(C, \overline{P Q})=3$ for any $Q \in C \cap \ell$. 
By this fact and $\phi\sigma(P)=P$, $\phi\sigma(F[P]\setminus\{P\})=F[P]\setminus \{P\}$. 
Since $\Delta' \cap F[P]\setminus \{P\}=\{P', P''\}$ and the automorphism $(X:Y:Z) \mapsto (X:Z:Y)$ is contained in $G_3(C)$, there exists $\phi_3 \in G_3(C)$ such that $\phi_3\sigma(P)=P$, $\phi_3\sigma(P')=P'$ and $\phi_3\sigma(P'')=P''$. 
Then $\phi_3\sigma$ is represented by the matrix of the form 
$$ \left(\begin{array}{ccc}
\alpha & 0 & 0 \\
0 & \beta & 0 \\
0 & 0 & 1 
\end{array}\right)$$
for some $\alpha, \beta \in K$. 
By considering the action on the defining equation, it follows that $\alpha^3=1$ and $\beta^3=1$. 
If we take 
$$
\phi_4=\left(\begin{array}{ccc}
\alpha^{-1} & 0 & 0 \\
0 & 1 & 0 \\
0 & 0 & 1 
\end{array}\right) \in G[P], \ 
\phi_5=\left(\begin{array}{ccc}
1 & 0 & 0 \\
0 & \beta^{-1} & 0 \\
0 & 0 & 1 
\end{array}\right) \in G[P'],$$ 
then $\phi_5\phi_4\phi_3\sigma=1$ on $\mathbb{P}^2$. 
Therefore, $\sigma=\phi_3^{-1}\phi_4^{-1}\phi_5^{-1} \in G_3(C)$. 
\end{proof}

\begin{remark}
For the curve defined by $X^6+Y^6+Z^6-10(X^3Y^3+Y^3Z^3+Z^3X^3)=0$, it is known that the group ${\rm Aut}(C)$ is isomorphic to the Hessian group of order $216$ (\cite{artebani-dolgachev}). 
\end{remark} 

\begin{theorem}
Let $C$ be the plane curve defined by $X^6+20X^3Y^3-8Y^6+Z^6=0$. 
Then there exist two exact sequences
$$ 0 \rightarrow \mathbb{Z}/6\mathbb{Z} \rightarrow G_3(C) \rightarrow A_4 \rightarrow 1, $$
$$ 0 \rightarrow \mathbb{Z}/6\mathbb{Z} \rightarrow {\rm Aut}(C) \rightarrow S_4 \rightarrow 1.  
$$
In particular, $|{\rm Aut}(C)|=144$ and $|G_3(C)|=72$. 
\end{theorem}

\begin{proof} 
Let $\ell$ be the line defined by $Z=0$, which contains $8$ points $P$ with $|G[P]|=3$. 
Since $\ell$ is a unique line containing $8$ points $P$ with $|G[P]|=3$, there exists a homomorphism $\varphi: {\rm Aut}(C) \rightarrow {\rm Aut}(\ell) \cong \mathbb{P}^1$.
Since $\varphi(G_3(C))=\varphi(\langle G[P_1], G[P_2]\rangle)$ for each points $P_1$ and $P_2$ such that $(P_1, P_2)$ is not a $G$-pair, it follows that  $\varphi(G_3(C)) \cong A_4$.  
Since $Q=(0:0:1)$ is a unique Galois point, the group ${\rm Aut}(C)$ fixes $Q$. 
This implies that ${\rm Ker} \ \varphi=G[Q] \cong \mathbb{Z}/6\mathbb{Z}$. 
The former exact sequence is obtained. 
On the other hand, the linear transformation 
$$ \left(\begin{array}{ccc}
0 & \sqrt{2} i & 0 \\
\frac{1}{\sqrt{2} i} & 0 & 0 \\
0 & 0 & 1 
\end{array} \right) $$
acts on $C$, where $i^2=-1$. 
This implies that $\varphi({\rm Aut}(C)) \cong S_4$. 
The latter exact sequence is obtained. 
\end{proof}

\section{Curves of degree four} 
In this section, we assume that $C$ is smooth and of degree $d=4$. 
The set $\Delta'_{\ge 2}$ of all quasi-Galois points in $\mathbb P^2 \setminus C$ for $C$ is denoted by $\Delta'$.  
If $P \in \Delta'$, then there exists a unique involution in $G[P]$, since $G[P]$ is a cyclic group of order $2$ or $4$. 
First, we note the following. 

\begin{lemma} \label{four cover}
If $P \in \Delta'$, then we have the following. 
\begin{itemize}
\item[(1)] There exist exactly four lines $\ell \ni P$ such that $C \cap \ell$ consists of one or two points, and the tangent line at each point of $C \cap \ell$ is equal to $\ell$. 
\item[(2)] There does not exist a line $\ell \ni P$ such that $I_Q(C, \ell)=3$ for some $Q \in C \cap \ell$.  
\end{itemize}
\end{lemma} 

\begin{proof}
Let $\sigma \in G[P]$ be the involution. 
The projection $\pi_P$ is the composite map of $g_P: C \rightarrow C/\sigma$ and $f_P:C/\sigma \rightarrow \mathbb P^1$. 
Since $g_P$ is ramified at exactly four points by Corollary \ref{fixed locus} and Fact \ref{Galois covering}(1), by Hurwitz formula, the genus of the smooth model of $C/\sigma$ is equal to $1$. 
Then $f_P: C/\sigma \rightarrow \mathbb P^1$ has exactly four ramification points. 
Therefore, we have (1). 
Assertion (2) is obvious, since $\pi_P$ is the composite map of double coverings $g_P$ and $f_P$. 
\end{proof}

We recall the notion of $G$-pairs and the following proposition (see Proposition \ref{pair 2} and Corollary \ref{pair 2'} in Section 3). 

\begin{proposition} \label{pair 3}
Let $(P, P')$ be a $G$-pair. 
Then there exists a linear transformation $\phi$ such that $\phi(P)=(1:0:0)$, $\phi(P')=(0:1:0)$, and $\phi(C)$ is given by 
$$ X^4+Y^4+Z^4+a X^2Y^2+b Y^2Z^2+c Z^2X^2=0, $$
where $a, b, c \in K$. 
In this case, there exists a quasi-Galois point $P''$ with $\phi(P'')=(0:0:1)$ such that $(P', P'')$ and $(P'', P)$ are $G$-pairs.  
In particular, $C \cap \overline{PP'}$ consists of exactly four points. 

Furthermore, if $P$ is a Galois point, then we can take $a=c=0$. 
\end{proposition} 

\begin{proof}
Assertions except for the last one are derived from Proposition \ref{pair 2} and Corollary \ref{pair 2'}. 
We consider the last assertion. 
Assume that $|G[P]|=4$. 
In the proof of Proposition \ref{pair 2}, we can take 
$$ B^{-1}A_{\sigma_1}B=\left(\begin{array}{ccc} i & 0 & 0 \\ 0 & 1 & 0 \\ 0 & 0 & 1 \end{array}\right), \ B^{-1}A_{\sigma_2}B=\left(\begin{array}{ccc} 1 & 0 & 0 \\ 0 & \zeta & 0 \\ 0 & 0 & 1 \end{array}\right), $$
where $i^2=-1$ and $\zeta^2=\pm 1$. 
Then the defining equation of the form 
$$ X^4+Y^4+Z^4+b Y^2Z^2=0$$
for some $b \in K$. 
\end{proof}

Let $\ell \subset \mathbb P^2$ be a projective line. 
We would like to calculate the number of quasi-Galois points on the line $\ell$.  
We treat the cases $\# C \cap \ell=4, 3, 2$ and $1$ separately. 

\begin{proposition} \label{four points} 
Let $\ell$ be a line with $\# C \cap \ell=4$. 
Then $\#\Delta' \cap \ell=0, 1, 2, 4$ or $6$.
Furthermore, if $\#\Delta' \cap \ell=2$ (resp., $4$, $6$), then we have exactly one (resp., two, three) $G$-pair. 
\end{proposition} 

\begin{proof}
Let $C \cap \ell=\{Q_1, Q_2, Q_3, Q_4\}$. 
We consider the possibilities of involutions acting on $C \cap \ell$. 
There are at most three types: 
\begin{itemize}
\item[(1)] $Q_1 \leftrightarrow Q_2$,\ $Q_3 \leftrightarrow Q_4$, 
\item[(2)] $Q_1 \leftrightarrow Q_3$,\ $Q_2 \leftrightarrow Q_4$, 
\item[(3)] $Q_1 \leftrightarrow Q_4$,\ $Q_2 \leftrightarrow Q_3$. 
\end{itemize}
If $P_1, P_2 \in \Delta' \cap \ell$, and involutions $\sigma_1 \in G[P_1]$ and $\sigma_2 \in G[P_2]$ are of type (1), then we have $\sigma_1|_\ell=\sigma_2|_\ell$. 
Then $\sigma_1(P_2)=\sigma_2(P_2)=P_2$ and $\sigma_2(P_1)=\sigma_1(P_1)=P_1$, i.e. $(P_1, P_2)$ is a $G$-pair. 
For each types (1)-(3) we have at most two quasi-Galois points, and hence, $\#\Delta' \cap \ell \le 6$.

Let $\sigma_1 \in G[P_1]$ and $\sigma_2 \in G[P_2]$ give involutions of types (1) and (2) respectively. 
Then $\sigma_1\sigma_2\sigma_1(Q_1)=Q_3$, and hence, $\sigma_1\sigma_2\sigma_1$ is of type (2). 
Since $\sigma_1(P_2) \ne P_2$ and $\sigma_1(P_2)$ is quasi-Galois, $(P_2, \sigma_1(P_2))$ is a $G$-pair. 
Similarly, $(P_1, \sigma_2(P_1))$ is a $G$-pair. 
We have two $G$-pairs. 

Assume that $\#\Delta'\cap\ell \ge 5$. 
There are at least two $G$-pairs. 
We can assume that $(P_1, P_2)$ and $(P_3, P_4)$ are $G$-pairs, and give involutions on $\ell$ of type (1) and (2) respectively. 
Let $P_5$ be another quasi-Galois point, and let $\sigma_i \in G[P_i]$ be the involution. 
Then $\sigma_1(P_5) \ne P_5$ and the involution $\sigma_1\sigma_5\sigma_1 \in G[\sigma_1(P_5)]$ gives an involution on $\ell$ of type (3). 
Therefore, $\sigma_1(P_5) \ne P_1, \ldots, P_5$. 
We have $\#\Delta'\cap\ell=6$.
\end{proof}

Hereafter, we consider the curve $C$ defined by 
$$ F=X^4+Y^4+Z^4+a X^2Y^2+b Y^2Z^2+c Z^2X^2=0. $$
Then $P=(1:0:0)$, $P'=(0:1:0)$, $P''=(0:0:1) \in \Delta'$. 
The lines $F[P]\setminus \{P\}$, $F[P']\setminus \{P'\}$ and $F[P''] \setminus \{P''\}$ are defined by $X=0$, $Y=0$ and $Z=0$ respectively. 
Since $C$ is smooth, we have $a \ne \pm 2$, $b \ne \pm 2$ and $c \ne \pm 2$.

\begin{proposition} \label{four-six points 1}
We have the following. 
\begin{itemize}
\item[(1)] If there exist two $G$-pairs on the line defined by $Z=0$, then $b=\pm c$. 
Furthermore, when $c=-b$, we take the linear transformation given by $X \mapsto iX$, where $i^2=-1$, so that we have the defining equation with $c=b$. 
\item[(2)] If there exist three $G$-pairs on the line defined by $Z=0$, then $b=c=0$. 
\end{itemize} 
\end{proposition} 

\begin{proof} 
Assume that $(P_1, P_1')$ and $(P_2, P_2')$ are two $G$-pairs on the line $Z=0$. 
Then the point $P_1''=(0:0:1)$ is contained in $F[P_1] \cap F[P_2]$. 
Let $\sigma_1\in G[P_1]$ and $\sigma_2 \in G[P_2]$ be involutions. 
Then $\sigma_1\sigma_2$ satisfies 
$$ P_1 \leftrightarrow P_1',\ P_2 \leftrightarrow P_2' $$
(see the second paragraph of the proof of Proposition \ref{four points}). 
Since $\sigma_1\sigma_2(F[P_1])=F[P_1']$, we have $\sigma_1\sigma_2(P_1'')=P_1''$. 
Then $\sigma_1\sigma_2$ is represented by the matrix 
$$ \left(\begin{array}{ccc} 0 & \lambda & 0 \\ 
\mu & 0 & 0 \\
0 & 0 & 1 
\end{array}\right) $$
for some $\lambda, \mu \in K$. 
Then $((\sigma_1\sigma_2)^{-1})^*F$ and $F$ are the same up to a constant. 
Therefore, we have  
$$ \lambda^4Y^4+\mu^4X^4+Z^4+a\lambda^2\mu^2X^2Y^2+b\mu^2X^2Z^2+c\lambda^2Y^2Z^2=F. $$
Considering the coefficients of $Y^4$ and $Y^2Z^2$, we have $\lambda^2=\pm 1$ and $b=\pm c$. 

Assume that $(P_1, P_1')$, $(P_2, P_2')$ and $(P_3, P_3')$ are three $G$-pairs on the line $Z=0$. 
Let $\sigma_2\in G[P_2]$ and $\sigma_3 \in G[P_3]$ be involutions. 
Then $\sigma_2\sigma_3$ satisfies 
$$ P_1 \rightarrow P_1,\ P_1' \rightarrow P_1', \ P_2 \leftrightarrow P_2',\ P_3 \leftrightarrow P_3' $$
(see the second paragraph of the proof of Proposition \ref{four points}). 
Since $\sigma_2\sigma_3(P_1)=P_1$ and $\sigma_2\sigma_3(P_1')=P_1'$, we have $\sigma_2\sigma_3(P_1'')=P_1''$. 
Note that the order of $\sigma_2\sigma_3$ is at least $3$ and the order of the restriction $(\sigma_2\sigma_3)|_{\{Z=0\}}$ on the line $Z=0$ is two.  
Then $\sigma_2\sigma_3$ is represented by the matrix 
$$ \left(\begin{array}{ccc} -\eta & 0 & 0 \\ 
0 & \eta & 0 \\
0 & 0 & 1 
\end{array}\right),$$
where $\eta^2 \ne 1$. 
Then $((\sigma_2\sigma_3)^{-1})^*F$ and $F$ are the same up to a constant. 
Therefore, we have  
$$ \eta^4X^4+\eta^4Y^4+Z^4+a\eta^4X^2Y^2+b \eta^2Y^2Z^2+c\eta^2Z^2X^2=F. $$
Considering the coefficients of $Y^2Z^2$ and $Z^2X^2$, we have $b=c=0$. 
\end{proof}

On the contrary, we have the following. 

\begin{proposition} \label{four-six points 2} 
Let $a, b \in K$, and let $C$ be the smooth plane curve given by 
$$ X^4+Y^4+Z^4+aX^2Y^2+bY^2Z^2+bZ^2X^2=0. $$
Then we have the following. 
\begin{itemize}
\item[(1)] Points $(1:0:0)$, $(0:1:0)$, $(1:1:0)$ and $(1:-1:0)$ are quasi-Galois points. Furthermore, if $b \ne 0$, they are not Galois. 
\item[(2)] If $b=0$, then points $(\pm i:1:0)$ are quasi-Galois, where $i^2=-1$. 
Furthermore, we have the following. 
\begin{itemize} 
\item{} If $a \ne 0$, then points $(1:0:0)$ and $(0:1:0)$ are not Galois. 
\item{} If $a \ne 6$, then points $(\pm 1:1:0)$ are not Galois.
\item{} If $a \ne -6$, then points $(\pm i:1:0)$ are not Galois. 
\item{} If $a =0$ or $\pm 6$, then there exists a linear transformation $\phi$ such that $\phi(\{Z=0\})=\{Z=0\}$ and $\phi(C)$ is the Fermat curve $X^4+Y^4+Z^4=0$. 
\end{itemize} 
\item[(3)] If $b=0$, then $\delta'[2]=6$ or $12$. 
Furthermore, $\delta'[2]=12$ if and only if $C$ is projectively equivalent to the Fermat curve $X^4+Y^4+Z^4=0$. 
\end{itemize} 
\end{proposition}

\begin{proof}
We consider points $(\pm 1:1:0)$.  
We set 
$$ \tilde{X}=\frac{1}{2}(X+Y),\ \tilde{Y}=\frac{1}{2}(X-Y),\ \tilde{Z}=Z$$
and take the linear transformation $\phi:(X:Y:Z) \mapsto (\tilde{X}:\tilde{Y}:\tilde{Z})$.
Then $\phi^{-1}((1:1:0))=(1:0:0)$, $\phi^{-1}((-1:1:0))=(0:1:0)$, and $\phi^{-1}(C)$ is given by 
$$ G=(2+a)X^4+(2+a)Y^4+Z^4+(12-2a)X^2Y^2+2bY^2Z^2+2bX^2Z^2=0. $$
By Theorem \ref{standard form}, $\phi^{-1}((\pm 1:1:0))$ are quasi-Galois. 
Therefore, $(\pm1:0:0)$ are quasi-Galois. 
Furthermore, if $\phi^{-1}((1:1:0))$ is Galois, then the matrix
$$ \left(\begin{array}{ccc} i & 0 & 0 \\
0 & 1 & 0 \\
0 & 0 & 1 \end{array}\right)$$
acts on $G$. 
This implies $12-2a=0$ and $2b=0$. 

Let $b=0$. 
We consider points $(\pm i:1:0)$.  
We set 
$$ \tilde{X}=\frac{1}{2}(X+iY),\ \tilde{Y}=\frac{1}{2}(X-iY),\ \tilde{Z}=Z$$
and take the linear transformation $\phi:(X:Y:Z) \mapsto (\tilde{X}:\tilde{Y}:\tilde{Z})$.
Then $\phi^{-1}((i:1:0))=(1:0:0)$, $\phi^{-1}((-i:1:0))=(0:1:0)$, and $\phi^{-1}(C)$ is given by 
$$ H=(2-a)X^4+(2-a)Y^4+Z^4+(12+2a)X^2Y^2=0. $$
By Theorem \ref{standard form}, $\phi^{-1}((\pm i:1:0))$ are quasi-Galois. 
Therefore, $(\pm i:0:0)$ are quasi-Galois. 
Furthermore, if $a \ne -6$, then $(\pm i:0:0)$ are not Galois. 

We prove (3). 
Now, we have six quasi-Galois points on the line $Z=0$. 
By the defining equation, we infer that $Q=(0:0:1)$ is an outer Galois point and the set $F[Q]\setminus \{Q\}$ is given by $Z=0$. 
Assume that $\delta'[2]>6$. 
Then there exists a quasi-Galois point $R \in \mathbb P^2 \setminus (\{Z=0\} \cup \{Q\})$. 
Let $\tau \in G[R]$ be the involution. 
If $\tau(Q)=Q$, then by Lemma \ref{pair 1}, $(R, Q)$ is a $G$-pair. 
Then $R$ must lie on the line $Z=0$. 
This is a contradiction. 
If $\tau(Q) \ne Q$ is not a $G$-pair, then we have two Galois points. 
It follows from a theorem of Yoshihara \cite{yoshihara} that $C$ is projectively equivalent to the Fermat curve. 
In this case, it is known that $\delta'[2]=12$ (\cite{miura-yoshihara, fmt}).  
\end{proof} 

\begin{corollary} \label{Galois point exist}
If $\delta'[\ge 2] \ge 2$ and $\delta'[4] \ge 1$, then there exists a line $\ell$ such that $\#\Delta' \cap \ell=6$. 
\end{corollary}

\begin{proof} 
If $\delta'[4] \ge 1$, then it follows from a theorem of Yoshihara \cite[Theorem 4' and Proposition 5']{yoshihara} that $\delta'[4]=1$ or $3$, and $\delta'[4]=3$ implies that $C$ is the Fermat curve. 
For the Fermat curve, the required line exists, by Proposition \ref{four-six points 2}. 
We can assume that $\delta'[4]=1$.  
Let $P$ be a Galois point and let $R$ be a quasi-Galois point. 
Since $\delta'[4]=1$, then $F[R] \ni P$, that is, $(R, P)$ is a $G$-pair. 
By Proposition \ref{pair 3}, the defining equation is of the form 
$$ X^4+Y^4+Z^4+b Y^2Z^2=0$$
for some $b \in K$. 
By Proposition \ref{four-six points 2}, the line defined by $X=0$ is the required line. 
\end{proof}

We consider the case where $C \cap \ell$ consists of three points. 

\begin{proposition} \label{three points}
If $\# C \cap \ell=3$, then $\#\Delta'\cap\ell=0$ or $1$. 
\end{proposition}

\begin{proof}
Let $C \cap \ell=\{Q_1, Q_2, Q_3\}$, let $T_{Q_1}C=\ell$, and let $P_1, P_2 \in \ell$ be different quasi-Galois points. 
Then $Q_1 \in C \cap F[P_1] \cap F[P_2]$. 
This is a contradiction to Proposition \ref{two quasi-Galois}. 
\end{proof}

We consider the case where $C \cap \ell$ consists of two points.

\begin{proposition} \label{two points} 
Let $C \cap \ell=\{Q_1, Q_2\}$, where $Q_1 \ne Q_2$. 
\begin{itemize} 
\item[(1)] If $I_{Q_1}(C, \ell)=3$, then $\#\Delta'\cap\ell=0$. 
\item[(2)] If $T_{Q_1}C=T_{Q_2}C=\ell$, then $\#\Delta'\cap\ell=0, 1$ or $3$.
\item[(3)] If $\#\Delta' \cap \ell=3$, then there exists an automorphism $\sigma \in {\rm Aut}(C)$ of order three such that the fixed locus of $\sigma$ coincides with the set $\{Q_1, Q_2, R\}$, where $R$ is the point given by $R \in F[P]$ for any $P \in \Delta' \cap \ell$. 
\end{itemize}
\end{proposition} 

\begin{proof}
Assertion (1) is derived from Lemma \ref{four cover}(2). 
We consider assertion (2). 
Let $P_1, P_2 \in \ell$ be quasi-Galois points, and let $\sigma_i \in G[P_i]$ be the involution. 
By Proposition \ref{pair 3}, $(P_1, P_2)$ is not a $G$-pair. 
By Lemma \ref{pair 1}, $\sigma_1(P_2) \ne P_2$, and hence, $\#\Delta'\cap\ell\ge 3$. 
We consider $\sigma_1\sigma_2$. 
Then $\sigma_1\sigma_2(Q_1)=Q_1$ and $\sigma_1\sigma_2(Q_2)=Q_2$. 
Let $R$ be the intersection point of the lines $F[P_1]\setminus \{P_1\}$ and $F[P_2]\setminus \{P_2\}$. 
Then $\sigma_1\sigma_2(R)=\sigma_1(R)=R$.  
If $R \in C$, then $T_RC \ni P_1, P_2$. 
Therefore, $R \not\in C$. 
For a suitable system of coordinates, we can assume that $Q_1=(1:0:0), Q_2=(0:1:0)$ and $R=(0:0:1)$. 
Consider the action of $\sigma_1\sigma_2$ on the lines $\overline{Q_1R}$ and $\overline{Q_2R}$. 
Since $\sigma_1\sigma_2$ fixes $Q_1, Q_2$ and $R$, $\sigma_1\sigma_2|_{\overline{Q_iR}}$ is identity if $\sigma_1\sigma_2$ fixes some point of $C \cap \overline{Q_iR}$ other than $Q_i$.
Therefore, the restriction $\sigma_1\sigma_2|_{\overline{Q_iR}}$ is identity or of order three for $i=1, 2$. 
Then $\sigma_1\sigma_2$ is represented by the matrix 
$$ A_{\sigma_1\sigma_2}=\left(\begin{array}{ccc} \zeta & 0 & 0 \\ 0 & \eta & 0 \\ 0 & 0 &1 \end{array} \right), $$
where $\zeta$ and $\eta$ are cubic roots of $1$. 
Since $Q_1$ and $Q_2$ are not inner Galois (by Facts \ref{index} and \ref{Galois covering}(2)) and $R \not\in C$, we have $\eta \ne 1$, $\zeta \ne 1$ and $\zeta \ne \eta$. 
This implies that $\eta=\zeta^2$.

Let $P_3:=\sigma_1\sigma_2(P_2)$ ($=\sigma_1(P_2)$). 
Since $\sigma_1\sigma_2(F[P_2])=F[P_3]$, $R \in F[P_3]$. 
Assume by contradiction that $\#\Delta'\cap\ell \ge 4$. 
Let $P_4 \ne P_1, P_2, P_3$ be quasi-Galois, and let $\sigma_4 \in G[P_4]$ be the involution. 
Since $\sigma_1\sigma_4(Q_i)=Q_i$ for $i=1, 2$ and the order of $\sigma_1\sigma_4$ is three, we have $\sigma_1\sigma_4|_{\ell}=\sigma_1\sigma_2|_{\ell}$ or $(\sigma_1\sigma_2)^2|_{\ell}$. 
Then we find that $\sigma_1\sigma_4(R)=R$, and hence, $\sigma_1\sigma_4=\sigma_1\sigma_2$ or $(\sigma_1\sigma_2)^2$ on $\mathbb P^2$. 
We have $\sigma_4=\sigma_2 \in G[P_2]$ or $\sigma_4=\sigma_2\sigma_1\sigma_2 \in G[P_3]$. 
This is a contradiction. 

The condition as in assertion (3) is satisfied for the automorphism $\sigma_1\sigma_2$. 
\end{proof}

We consider the case where $C \cap \ell$ consists of a unique point.

\begin{proposition} \label{one point} 
If $\# C \cap \ell=1$, then $\#\Delta'\cap\ell=0$ or $1$. 
\end{proposition} 

\begin{proof}
Let $C \cap \ell=\{Q\}$, and let $P_1, P_2$ be different quasi-Galois points. 
Then $Q \in F[P_1] \cap F[P_2]$. 
This is a contradiction to Proposition \ref{two quasi-Galois}. 
\end{proof}

Here, we assume that there does not exist a line $\ell$ such that $\#\Delta' \cap \ell=6$.
This condition is equivalent to the one that there does not exist a Galois point, under the assumption that $\delta'[\ge 2]\ge 2$, by Proposition \ref{four-six points 2} and Corollary \ref{Galois point exist}.   
We introduce the notion of ``$G$-triple'' and ``$G$-triangle'' here. 
We call a triple $(P, P', P'')$ a $G$-triple, if each two of points $P, P', P''$ form a $G$-pair. 
We call the set $\overline{P P'} \cup \overline{P' P''} \cup \overline{P'' P} \subset \mathbb{P}^2$ a $G$-triangle via the triple $(P, P', P'')$. 

\begin{lemma} \label{not in triangle}
Let $(P, P', P'')$ be a $G$-triple. 
Assume that $R$ is a quasi-Galois point not in the $G$-triangle $\overline{PP'}\cup \overline{P' P''} \cup \overline{P'' P}$. 
Then one of three lines $\overline{R P}$, $\overline{RP'}$ and $\overline{R P''}$ is not a multiple tangent line, that is, one of them contains at least three points of $C$.
\end{lemma}

\begin{proof}
For a suitable system, we can assume that $P=(1:0:0)$, $P'=(0:1:0)$, $P''=(0:0:1)$ and $R=(1:1:1)$. 
Then the defining equation of $C$ is of the form 
$$ F=X^4+Y^4+Z^4+a X^2Y^2+b Y^2Z^2+c Z^2X^2=0, $$
where $a, b, c \in K$ and $F(1, 1, 1)=3+a+b+c \ne 0$. 
Assume that the three lines are multiple tangent lines. 
By the condition that the line $\overline{RP}$ is a multiple tangent line, it follows that
$$ D_1(1, 1)=(a+c)^2-4(b+2)=0, $$
where $D_1(Y, Z)$ is the discriminant 
$$ (aY^2+cZ^2)^2-4(Y^4+bY^2Z^2+Z^4). $$
By the symmetry, we have the equations
$$ D_2=(a+b)^2-4(c+2)=0, \mbox{ and } D_3=(b+c)^2-4(a+2)=0. $$
By the relation $D_1-D_2=0$, 
$$ (b-c)(2a+b+c+4)=0. $$
Similarly, 
$$ (a-b)(a+b+2c+4)=0, \mbox{ and } (c-a)(a+2b+c+4)=0. $$

Assume that $a=b=c$. 
Then $3a+3 \ne 0$ and $(2a)^2-4(a+2)=0$. 
This implies that $a=2$. 
This is a contradiction to the smoothness. 

We can assume that $a \ne b$. 
Then $a+b+2c+4=0$. 
If $b=c$, then $a=-3c-4$ and $(-2c-4)^2-4(c+2)=0$. 
Then $c=-2$ or $a=b=c=-1$. 
The former is a contradiction to the smoothness, and the latter is a contradiction to $a \ne b$. 
Therefore, $b \ne c$. 
Then 
$$ 2a+b+c+4=a+b+2c+4=0,  $$
and $a=c$. 
Therefore, $b=-3c-4$ and $(-2c-4)^2-4(c+2)=0$. 
Then $c=-2$ or $c=-1$. 
The former is a contradiction to the smoothness, and the latter is a contradiction to $b \ne c$. 
\end{proof}

\begin{proposition} \label{triple}
Assume that there exists a $G$-triple $(P, P', P'')$, and $\delta'[4]=0$. 
Then 
$$ \delta'[2]=3, 5, 9 \mbox{ or } 21. $$
Furthermore, the following hold. 
\begin{itemize}
\item[(1)] $\delta'[2]=21$ if and only if $C$ is projectively equivalent to the curve defined by 
$$X^4+Y^4+Z^4+a(X^2Y^2+Y^2Z^2+Z^2X^2)=0, $$
where $a \in K$ satisfies $a^2+3a+18=0$. 
\item[(2)] $\delta'[2]=9$ if and only if $C$ is projectively equivalent to the curve defined by 
$$ X^4+Y^4+Z^4+a(X^2Y^2+Y^2Z^2+Z^2X^2)=0, $$
where $a \in K \setminus \{0, -1\}$ and $a^2+3a+18 \ne 0$. 
\item[(3)] $\delta'[2]=5$ if and only if $C$ is projectively equivalent to the curve defined by 
$$ X^4+Y^4+Z^4+aX^2Y^2+b(Y^2Z^2+Z^2X^2)=0, $$
where $a, b \in K$, $b \ne 0$ and $b \ne \pm a$.  
\end{itemize} 
\end{proposition} 

\begin{proof}
Assume that there exists a quasi-Galois point $R \not \in \overline{PP'} \cup \overline{P' P''} \cup \overline{P'' P}$. 
By Proposition \ref{not in triangle}, we can assume that $\overline{R P''}$ is not a multiple tangent line. 
By Propositions \ref{four points}, \ref{three points} and \ref{two points}, it follows that there exist four quasi-Galois points on the line $\overline{RP''}$, and that the point $P_2$ given $\overline{PP'} \cap \overline{RP''}$ is a quasi-Galois point.  
For the involution $\sigma \in G[P]$, $\sigma(P_2)$ is a quasi-Galois point and the line $\overline{\sigma(P_2) P''}$ contains four quasi-Galois points. 
It follows that the triple $(P_2, \sigma(P_2), P'')$ is a $G$-triple such that two edges of the $G$-triangle contain four quasi-Galois points. 
In this case, it follows form Proposition \ref{four-six points 1} that $C$ is projectively equivalent to the curve defined by 
$$ X^4+Y^4+Z^4+a(X^2Y^2+Y^2Z^2+Z^2X^2)=0, $$
and the triangle $\overline{P_2\sigma(P_2)} \cup \overline{\sigma(P_2)P''} \cup \overline{P'' P_2}$ contains $9$ quasi-Galois points. 
By taking a suitable system of coordinates, we can assume that $P=(1:0:0)$, $P'=(0:1:0)$, $P''=(0:0:1)$, $C$ is defined by 
$$X^4+Y^4+Z^4+a(X^2Y^2+Y^2Z^2+Z^2X^2)=0 $$
for some $a \in K \setminus \{0\}$, and $\#\Delta' \cap (\overline{PP'} \cup \overline{P' P''} \cup \overline{P'' P})=9$. 

Assume that $\delta'[2] \ge 10$. 
Let $R$ be a quasi-Galois point with $R \not \in \overline{PP'} \cup \overline{P' P''} \cup \overline{P'' P}$. 
By Proposition \ref{not in triangle}, we can assume that $\overline{R P}$ is not a multiple tangent line. 
It follows from Proposition \ref{four-six points 2} that $P_2:=(0:1:1)$ and $P_2':=(0:-1:1) \in F[P]\setminus \{P\}=\{X=0\}$ are quasi-Galois. 
Note that $(P_2, P_2')$ is a $G$-pair, since $(P', P'')$ is a $G$-pair. 
We can assume that $R \in \overline{P P_2'}$ with $R \ne P, P_2'$. 
Let $\tau \in G[R]$ be the involution. 
Note that $\tau(P_2')=P$, since $(P, P_2')$ is a $G$-pair. 
Since $(P, P_2)$ and $(P_2, P_2')$ are $G$-pairs, $F[P_2]\setminus \{P_2\}=\overline{P P_2'} \ni R$, and hence, $(P_2, R)$ is a $G$-pair. 
Therefore $\tau(P_2)=P_2$. 
Since $\tau((0:1:1))=(0:1:1)$, $\tau((0:-1:1))=(1:0:0)$, and $\tau((1:0:0))=(0:-1:1)$, $\tau$ is represented by the matrix 
$$ \left(\begin{array}{ccc}
0 & \frac{2}{\lambda} & -\frac{2}{\lambda} \\
\lambda & 1 & 1 \\
-\lambda & 1 & 1 
\end{array}\right), $$
where $\lambda \in K$. 
Then $(\tau^{-1})^*F$ and $F$ are the same up to a constant. 
Here 
\begin{eqnarray*} 
(\tau^{-1})^*F&=&\left(\frac{2}{\lambda}\right)^4(Y-Z)^4+(\lambda X+Y+Z)^4+(-\lambda X+Y+Z)^4\\ & &+a\left(\frac{2}{\lambda}\right)^2(Y-Z)^2(\lambda X+Y+Z)^2+a(\lambda X+Y+Z)^2(-\lambda X+Y+Z)^2 \\ & & +a\left(\frac{2}{\lambda}\right)^2(-\lambda X+Y+Z)^2(Y-Z)^2. 
\end{eqnarray*}
The coefficient of $X^2YZ$ is 
$$ 12\lambda^2+12\lambda^2-2a\left(\frac{2}{\lambda}\right)^2\lambda^2-4a\lambda^2-2a\left(\frac{2}{\lambda}\right)^2\lambda^2=0. $$
We have $\lambda^2=4a/(6-a)$. 
The coefficient of $Y^3Z$ is 
$$ -4\left(\frac{2}{\lambda}\right)^4+4+4+4a=0. $$
We have $a^3+a^2+12a-36=0$.  
Since $a \ne 2$, we have $a^2+3a+18=0$.

On the contrary, let $a^2+3a+18=0$, and let $C$ be the plane curve given by 
$$ F=X^4+Y^4+Z^4+a(X^2Y^2+Y^2Z^2+Z^2X^2)=0.  $$
By Proposition \ref{four-six points 2}, we have $9$ quasi-Galois points on the union of the lines $X=0$, $Y=0$ and $Z=0$. 
Let $\lambda$ be a solution of $\lambda^2=4a/(6-a)$, and let $\tau$ be the involution given by
$$ \left(\begin{array}{ccc}
0 & \frac{2}{\lambda} & -\frac{2}{\lambda} \\
\lambda & 1 & 1 \\
-\lambda & 1 & 1 
\end{array}\right).  
$$
Then $\tau$ acts on $C$. 
It is inferred that $\tau$ is an involution, $\tau(\{Y+Z=0\})=\{Y+Z=0\}$, and $\tau$ is not identity on this line. 
By the proof of \cite[Proposition 2.6]{fmt}, $\tau$ is the involution of some quasi-Galois point $R$ on the line $Y+Z=0$ other than $(1:0:0)$ or $(0:-1:1)$. 
By considering the actions associated with points $(1:0:0)$, $(0:1:0)$ and $(0:0:1)$, we have four quasi-Galois points not in $\{XYZ=0\}$. 
Note that $R$ is different from $(\pm1:\pm1:1)$. 
Using the linear transformation given by $(X:Y:Z) \mapsto (Z:X:Y)$, we have $4 \times 3$ additional quasi-Galois points. 
Therefore, we have $\delta'[2] \ge 9+12=21$. 

We prove that $\delta'[2] \le 21$. 
Let $P \in \Delta'$ and let $\ell \ni P$ be a line. 
If $\#\Delta' \cap \ell \ge 2$, then $\#\Delta' \cap \ell=2, 3$ or $4$, by Propositions \ref{four points}, \ref{three points}, \ref{two points} and \ref{one point}.  
If $\#\Delta' \cap \ell=3$, then $\ell$ is a multiple tangent line, and hence, it follows from Lemma \ref{four cover} that such lines $\ell$ are at most $4$. 
If $\#\Delta' \cap \ell=2$ or $4$, then there exists a point $P' \in \ell$ such that $P' \in F[P] \setminus \{P\}$, by Proposition \ref{four points}. 
Since $\#\Delta' \cap (F[P] \setminus \{P\}) \le 4$, such lines $\ell$ are at most $4$. 
Therefore, $\delta'[2] \le 1+4 \times 2+4 \times 3=21$.  

Assume that $\Delta' \subset \overline{PP'} \cup \overline{P' P''} \cup \overline{P'' P}$ and $\delta'[2] \ge 4$.
Let $R \in \Delta' \setminus \{P, P', P''\}$.  
We can assume that $R \in \overline{P P'}$.  
By Proposition \ref{four points}, there exist four quasi-Galois points on the line $\overline{PP'}$. 
For a suitable system of coordinates, we can assume that the line $\overline{P P'}$ is defined by $Z=0$ and $C$ is defined by 
$$ X^4+Y^4+Z^4+a X^2Y^2+b(Y^2Z^2+Z^2X^2)=0, $$
where $a, b \in K$. 
By Proposition \ref{four-six points 2}, it follows that $b \ne 0$. 
In this case, $\delta'[2] \ge 5$. 
If $\delta'[2] >5$, then the line $\overline{P' P''}$ or $\overline{P''P}$ contains four quasi-Galois points. 
It follows from Proposition \ref{four-six points 1} that $a=\pm b$. 
In this case, $C$ is projectively equivalent to the curve defined by 
$$ X^4+Y^4+Z^4+a(X^2Y^2+Y^2Z^2+Z^2X^2)=0 $$
where $a \ne 0$ and $a^2+3a+18 \ne 0$, and $\delta'[2]=9$. 
For the smoothness, we need the condition $a \ne -1$. 
If $\delta'[2]=5$, then $b \ne \pm a$, by Proposition \ref{four-six points 1}.  
\end{proof}

\begin{theorem} \label{quartic} 
Let $C \subset \mathbb P^2$ be a smooth curve of degree four. 
Then
$$ \delta'[2]=0, 1, 3, 5, 6, 9, 12 \mbox{ or } 21. $$
Furthermore, the following hold. 
\begin{itemize}
\item[(1)] $\delta'[2]=21$ if and only if $C$ is projectively equivalent to the curve defined by 
$$X^4+Y^4+Z^4+a(X^2Y^2+Y^2Z^2+Z^2X^2)=0, $$
where $a \in K$ satisfies $a^2+3a+18=0$. 
\item[(2)] $\delta'[2]=12$ if and only if $C$ is projectively equivalent to the Fermat curve. 
\item[(3)] $\delta'[2]=9$ if and only if $C$ is projectively equivalent to the curve defined by 
$$ X^4+Y^4+Z^4+a(X^2Y^2+Y^2Z^2+Z^2X^2)=0, $$
where $a \in K \setminus \{0, -1\}$ and $a^2+3a+18 \ne 0$. 
\item[(4)] $\delta'[2]=6$ if and only if $C$ is projectively equivalent to the curve defined by 
$$ X^4+Y^4+Z^4+aX^2Y^2=0, $$
where $a \in K \setminus\{0\}$ and $a \ne \pm 6$. 
\item[(5)] $\delta'[2]=5$ if and only if $C$ is projectively equivalent to the curve defined by 
$$ X^4+Y^4+Z^4+aX^2Y^2+b(Y^2Z^2+Z^2X^2)=0, $$
where $a, b \in K$, $b \ne 0$ and $b \ne \pm a$.  
\end{itemize} 
\end{theorem} 

\begin{proof}
Assume that $\delta'[\ge 2]\ge 2$ and $\delta'[4] \ge 1$. 
By Proposition \ref{four-six points 2} and Corollary \ref{Galois point exist}, $C$ is projectively equivalent to the curve defined by 
$$ X^4+Y^4+Z^4+aX^2Y^2=0. $$
Furthermore, if $a=0$ or $\pm 6$, then $\delta'[2]=12$ and $C$ is the Fermat curve, otherwise, $\delta'[2]=6$. 
The assertion follows. 

Hereafter, we assume that $\delta'[4]=0$ and $\delta'[2] \ge 2$.  
By Proposition \ref{four-six points 2},  there does not exist a line $\ell$ such that $\#\ell \cap \Delta'=6$. 
If two Galois points form a $G$-pair, then, by Proposition \ref{pair 3}, there exists a $G$-triple and $\delta'[2] \ge 3$. 
Therefore, $\delta'[2] \ge 3$ in any case. 
If there exists a $G$-triple, then the assertion follows by Proposition \ref{triple}. 
We can assume that there does not exist a $G$-pair. 

Assume that $\delta'[2] \ge 4$. 
By Propositions \ref{four points}, \ref{three points}, \ref{two points} and \ref{one point}, there exists a line $\ell$ containing exactly three quasi-Galois points $P_1, P_2$ and $P_3$. 
Let $P_4$ be a quasi-Galois point with $P_4 \not\in \ell$. 
Then the line $\overline{P_iP_4}$ contains exactly three quasi-Galois points for each $i$. 
Therefore, $\delta'[2] \ge 7$. 
It follows from Proposition \ref{two points} that there exists an automorphism $\sigma$ of order three such that $\sigma(\ell)=\ell$ and the fixed point of $\sigma$ not in $\ell$ coincides with the point given by $F[P_1] \cap F[P_2]$. 
If $\delta'[2]=7$ or $8$, then $\sigma$ acts on $7-3=4$ or $8-3=5$ points.  
Therefore, $\sigma$ fixes some quasi-Galois point $P$. 
Then $P \in F[P_1]$ and hence, $(P, P_1)$ is a $G$-pair. 
This is a contradiction. 
It follows that $\delta'[2] \ge 9$.
 
Assume that $\delta'[2]=9$. 
Since there does not exist a $G$-pair, the line $\overline{P_1P_2}$ contains exactly three quasi-Galois points for each pair of different quasi-Galois points $P_1$ and $P_2$. 
We consider the set 
$$ I:=\{(P, \ell) \in \Delta' \times \check{\mathbb{P}}^{2} \ | \ P \in \ell, \ \#\Delta' \cap \ell =3\} $$
with projections $p_1: I \rightarrow \Delta'$ and $p_2: I \rightarrow \check{\mathbb{P}}^2$, where $\check{\mathbb{P}}^{2}$ is the dual projective plane. 
Since $\#I=9\times4=36$ and each fiber of $p_2$ contains exactly $3$ points, it follows that $\#p_2(I)=12$. 
Let $\ell \in p_2(I)$ and let $\sigma$ be an automorphism of order three on the line $\ell$ as in Proposition \ref{two points}. 
Since the automorphism $\sigma$ acts on the set $p_2(I) \setminus \{\ell\}$ and $\#p_2(I) \setminus \{\ell\}=11$, there exists a line $\ell' \in p_2(I) \setminus \{\ell\}$ such that $\sigma(\ell')=\ell'$.  
Then $\sigma(\ell \cap \ell')=\ell \cap \ell'$. 
By Proposition \ref{two points}, the point given by $\ell \cap \ell'$ is contained in $C$. 
Since the tangent line at the point given by $\ell \cap \ell'$ coincides with lines $\ell$ and $\ell'$. 
This is a contradiction to the smoothness. 
It follows that $\delta'[2] \ge 10$. 

Assume that $\delta'[2] \ge 10$. 
Let $P \in \Delta'$. 
Since lines containing $P$ and another two quasi-Galois points are at most $4$, there exists a line containing four quasi-Galois points.  
In this case, there exists a $G$-pair, by Proposition \ref{four points}. 
This is a contradiction. 
\end{proof}

\begin{remark} 
It is known that the curve defined by 
$$ X^4+Y^4+Z^4+a(X^2Y^2+Y^2Z^2+Z^2X^2)=0, $$
where $a \in K$ satisfies $a^2+3a+18=0$, is projectively equivalent to the Klein quartic $X^3Y+Y^3Z+Z^3X=0$ (\cite{klein}, \cite{rodriguez-gonzalez}). 
\end{remark}

\begin{remark}
For $d=5$ and $n=2$, the third author determined the number $\delta[2]$ (\cite{takahashi}). 
\end{remark}

\end{document}